%% file: PWT.tex
\documentclass[reqno]{amsart}
\usepackage{amsthm} 
\usepackage{amsmath}
\usepackage{amssymb}
\usepackage{comment}
\usepackage{graphicx}
\usepackage{subcaption}
\usepackage{subfloat}
\usepackage{enumitem}
\usepackage{tikz}
\usetikzlibrary{arrows}
\usetikzlibrary{decorations.markings}
\usetikzlibrary{decorations.pathreplacing,calligraphy}
\usetikzlibrary{decorations.pathmorphing}
\usetikzlibrary{decorations.pathreplacing,shapes.misc}
\usepackage[dvipsnames]{xcolor}
\usepackage[section]{placeins}
\usepackage[mathscr]{euscript}
\usepackage{physics}
\usepackage{caption}
\usepackage{subcaption}
\usepackage{wrapfig}
\usepackage{hyperref}
\usepackage{stmaryrd}

\newtheorem{theorem}{Theorem}[section]
\newtheorem{proposition}[theorem]{Proposition}
\newtheorem{corollary}[theorem]{Corollary}
\newtheorem{lemma}[theorem]{Lemma}

\theoremstyle{definition}
\newtheorem{definition}[theorem]{Definition}
\newtheorem{example}[theorem]{Example}
\theoremstyle{remark}
\newtheorem*{remark}{Remark}

\theoremstyle{plain}
\newtheorem{maintheorem}{Theorem}

\newcommand{\ZZ}{\mathbb{Z}}
\newcommand{\NN}{\mathbb{N}}

\newcommand{\CC}{\mathcal{C}}
\newcommand{\GG}{\mathcal{G}}

\newcommand{\HH}{\mathcal{H}}

\newcommand{\EE}{\mathcal{E}}
\newcommand{\LL}{\mathcal{L}}
\newcommand{\VV}{\mathcal{V}}

\DeclareMathOperator{\cy}{Cay}

\DeclareMathOperator{\dpt}{dp}
\DeclareMathOperator{\stb}{Stab}
\DeclareMathOperator{\stbr}{\stb_{\circlearrowright}}
\DeclareMathOperator{\dom}{dom}
\DeclareMathOperator{\inv}{inv}
\DeclareMathOperator{\id}{id}

\newcommand{\cay}{\cy(G,S)}

\newcommand{\cayda}{\cy(D,A)}
\newcommand{\caydsx}{\cy(D,S(X))}

\newcommand{\refh}{\langle \rho_H \rangle}
\newcommand{\HHi}{\HH^{\inv}}
\newcommand{\caycx}{\cy(C,X_2)}
\newcommand{\tm}{\Tilde{m}}
\newcommand{\tgg}{\Tilde{\GG}}
\newcommand{\tf}{\Tilde{f}}
\newcommand{\td}{\Tilde{D}}
\newcommand{\tX}{\Tilde{X}}
\newcommand{\tx}{\Tilde{x}}
\newcommand{\teta}{\Tilde{\eta}_v}

\definecolor{ashgrey}{rgb}{0.7, 0.75, 0.71}

\title{Dyer Groups have the falsification by fellow-traveller property}
\author{Megan Howarth}

\address{University of Geneva, Section of Mathematics, Rue du Conseil-Général 7-9, 1205 Geneva, Switzerland.}
\email{Megan.Howarth@unige.ch}



\begin{document}

\vspace{-1em}

\begin{abstract}
  This paper is devoted to the study of the falsification by fellow-traveller property (FFTP) in Dyer groups. We exhibit a finite generating set for which the associated Cayley graph is a locally finite mediangle graph, and leverage its properties to prove that Dyer groups have the FFTP with respect to this generating set, from which it follows that Dyer groups have finitely many cone types. A key component of the proof is a generalisation of the Parallel Wall Theorem from Coxeter groups to Dyer groups. These results emphasise the role of Dyer groups in providing a unified approach to Coxeter groups and graph products of cyclic groups.
\end{abstract}

\maketitle

\section{Introduction}

\textit{Dyer groups} form a class of groups which first appeared in Dyer's work on reflection subgroups of Coxeter systems \cite{dyer_reflection_1990} in the $90$s, and have since become the subject of independent study \cite{PaS, PaV}; additionally, they are instances of \textit{periagroups} \cite{genevois, Genevois2, soergel_dyer_2024}. The main appeal of Dyer groups lies in the fact that they provide a unified framework for studying both Coxeter groups and graph products of cyclic groups, which include right-angled Artin groups (RAAGs). Indeed, they are all defined by (vertex- and/or edge-) labelled graphs, and it has been shown that they share many fundamental properties, including the rationality of their growth series \cite{PaV} and a common solution to the word problem \cite{PaS}. \\

By methods specific to each family, it is also known 
that Coxeter groups and graph products of cyclic groups, equipped with their standard generating sets, form Cannon pairs \cite{brinkhowlett, LMW}; that is, they admit only \textit{finitely many cone types}. The notion of \textit{cone types} was introduced by Cannon in the $80$s; intuitively, the cone type of an element corresponds to the set of elements which extend it geodesically. He used them to compute growth functions of surface groups and some triangle groups \cite{cannonpreprint, cannon_growth_1992}, and ultimately to prove rationality of growth functions for all Gromov hyperbolic groups \cite{Cannon}. Notably, a finite set of cone types enables the study of infinite objects with only finite information and provides insight into their asymptotic geometry. A direct consequence of being a Cannon pair is that the full language of geodesics is \textit{regular}, that is, recognised by a finite-state automaton \cite{Antolin, NS-course}. 
Beyond hyperbolic groups, which have finitely many cone types with respect to \textit{any} finite generating set, there are some classes of groups which have been shown to form Cannon pairs with respect to \textit{some} generating set, including Coxeter groups \cite{brinkhowlett},  wallpaper groups \cite{grigorchuk_growth_2022}, and more generally virtually abelian groups \cite{NS}. Moreover, this finiteness property is preserved under taking graph products \cite{LMW}. As an application, cone types can be leveraged to better understand the asymptotic characteristics of the group, including to compute the (rational) growth series \cite{grigorchuk_growth_2022, HowNag}, or tight upper and lower estimates for the spectral radius of the simple random walk on its Cayley graph \cite{gouezel_numerical_2015, HowNag, nagnibeda_estimate_1999, Nagnibeda}, in the non-amenable case. \\

In a similar direction, the \textit{falsification by fellow-traveller property (FFTP)} was initially introduced in the $90$s by Neumann and Shapiro, inspired by ideas of Cannon \cite{Cannon}, to approach the problem of determining which pairs of groups together with a finite generating set admit a regular language of geodesics. We note that the FFTP implies the finiteness of the cone types, and is in fact strictly stronger \cite{elder1, NS-course}. Additionally, the FFTP has been shown to imply many other interesting properties, such as almost convexity, finite presentation, at most quadratic isoperimetric function, type $F_3$ \cite{elder2} and solvable word problem in quadratic time \cite{holtgarside}. Among groups known to satisfy this property with respect to a suitable generating set, we find geometrically finite hyperbolic groups \cite{NS}, virtually abelian groups \cite{NS},   
Coxeter groups \cite{noskov}, Artin groups of large type \cite{holt_artin_2012}, dihedral Artin groups \cite{ciobanu_conjugacy_2025} 
and Garside groups \cite{holtgarside}, which include braid groups and Artin groups of finite type \cite{dehornoy_gaussian_1999}. \\

Dyer groups come equipped with a standard generating set $X^{\pm 1}$, for which the FFTP is not known to hold in general. In this paper, we introduce a finite generating set $X_S$, which is a modification of the set of \textit{syllables} $S(X)$ -- another natural generating set for Dyer groups (Definition \ref{def:gen-sets}).
Our main result is that the pair $(D,X_S)$ has the FFTP.
\begin{maintheorem}[Corollary \ref{cor:FFTP}] \label{thm:intro-fftp}
Let $D=D(\GG,m,f)$ be a Dyer group. Then $(D,X_S)$ has the FFTP with constant $2\Tilde{Q}+2$, where $\Tilde{Q}$ is the Parallel Wall constant for the infinite-vertex double $\td$.
\end{maintheorem}

A consequence of this theorem is that $(D,X_S)$ is a Cannon pair (Corollary \ref{cor:cannon-pair}), and therefore its language of geodesics is regular. 
In the case of Coxeter groups, we retrieve Noskov's theorem \cite{noskov}. For RAAGs, $X^{\pm 1} = X_S$, so the result holds for the standard generating set, and more generally for graph products of cyclic groups, the FFTP constant is $2$ (see Subsection \ref{subsec:right-hyps}).

The proof of this theorem has two key components. We first generalise the Parallel Wall Theorem for Coxeter groups to the setting of Dyer groups (taken with the syllabic generating set); see Section \ref{sec:PWT-dyer} for the relevant definitions. The statement is:

\begin{maintheorem} [Theorem \ref{thm:pwt-dyer}] \label{thm:intro-pwt}
     Let $D=D(\GG,m,f)$ be a Dyer group. There exists a constant $Q=Q(\GG,m,f)<\infty$ such that for every $g \in D$ and hyperplane $H$ of $\caydsx$ such that $d(g, M(H)) > Q$, there exists a hyperplane which is non-transverse to $H$ and separates $g$ from $H$.
\end{maintheorem}

As $\caydsx$ may not be a locally finite graph, we define an operation on the Dyer graph (Definition \ref{def:infinite-doubling}) in such a way that the resulting Dyer group $\td$ has a finite set of syllabic generators. We show:

\begin{maintheorem}[Theorem \ref{thm:cay-xs-sx}] \label{thm:intro-iso}
    Let $D$ be a Dyer group and $\td$ its infinite-vertex double. Then,
        $\cy(D,X_S)$ and $\cy(\td, S(\tX))$
    are isomorphic as graphs.
\end{maintheorem}

Our main tools are the geometric description of the Cayley graphs of Dyer groups as \textit{mediangle graphs}, introduced by Genevois \cite{genevois, Genevois2}, together with the structural properties of their hyperplanes, which we show to correspond to reflections. \\

This paper is organised as follows. We start with Section \ref{sec:background}, where we describe the framework and give the necessary background on Coxeter groups and Dyer groups, (mediangle) graphs, cone types and the falsification by fellow-traveller property. In Section \ref{sec:PWT-dyer}, we study the link between the reflections and hyperplanes in the Cayley graph of a Dyer group, recall the Parallel Wall Theorem for Coxeter groups and develop the tools to generalise it to the setting of Dyer groups. Finally, in Section \ref{sec:fftp}, we introduce the finite generating set $X_S$ and the doubling operation, and prove that Dyer groups have the FFTP, hence finitely many cone types.

\section{Background and preliminary notions}
\label{sec:background}

In this section, we collect the definitions and basic results about Coxeter groups, Dyer groups, cone types and the falsification by fellow-traveller property that are relevant for the rest of this work.

\subsection{Coxeter groups and root systems} 
\label{subsec:coxeter}

\begin{definition} \label{def:coxeter}
Let $\GG=(V(\GG),E(\GG))$ be a finite simplicial graph. We endow $E(\GG)$ with the following edge-labelling map\footnote{Note that non-adjacent vertices $u,v$ impose no relation; this is sometimes denoted $m(\{u,v\}) = \infty$.}
\begin{align*}
    m : E(\GG) \to \NN_{\geq 2}.
\end{align*}
The couple $(\GG,m)$ is called a \textit{Coxeter graph} and the group $W = W(\GG,m)$ associated to the data $(\GG,m)$ is called a \textit{Coxeter group} and is defined by the \textit{standard presentation}
 \begin{equation*}
    \label{eq:coxpres}
        W = \Biggl\langle 
       \begin{array}{l|cl}
                      x_v, v \in V(\GG)  & x_v^2 = 1 \hspace{0.5em} \forall v \in V(\GG), \\
             & [x_u,x_v]_{m(e)} = [x_v,x_u]_{m(e)} \hspace{0.5em} \forall e = \{u,v\} \in E(\GG)
        \end{array}
     \Biggr\rangle,
    \end{equation*}
where $[a,b]_m:= \underbrace{abab \ldots}_{m}$ denotes the alternating product of $a$ and $b$ of length $m$.
\end{definition}
We denote by $S:=\{x_v \mid v \in V(\GG) \}$ the \textit{standard Coxeter generating set}. We refer to the Coxeter group $W$ with this fixed generating set $S$ as a \textit{Coxeter system}, denoted $(W,S)$.

\begin{definition}
    The \textit{set of reflections} of a Coxeter group is the set of conjugates of its standard generators
    \begin{align*}
        R := \{wsw^{-1} \mid w \in W, s\in S\}.
    \end{align*}
\end{definition}

Recall that one can construct a geometric representation of Coxeter groups. Consider a real vector space $\VV$ of dimension $\lvert S \lvert$; we fix the basis $\{\alpha_s \mid s \in S\}$ and endow it with a symmetric bilinear form 
\begin{align*}
    B(\alpha_s, \alpha_t) = \begin{cases}
        1 & \text{if } s=t, \\
        -\cos \left(\frac{\pi}{m(\{u,v\})} \right) & \text{if } s=x_u, t=x_v \text{ and } \{u,v\} \in E(\GG), \\
        -1, & \text{if } s=x_u, t=x_v \text{ and } \{u,v\} \notin E(\GG).
    \end{cases}
\end{align*}
Each $s \in S$ acts on $\VV$ by the reflection
\begin{align*}
    \sigma_s(x) = x - 2 B(\alpha_s, x) \alpha_s,
\end{align*}
which extends to a faithful representation of $W$ on $\VV$; see \cite{bjorner_combinatorics_2005} for more details.

\begin{definition}
Let $(W,S)$ be a Coxeter system. We define the following sets:
\begin{enumerate}
    \item  $\Pi := \{\alpha_s \mid s \in S\}$ the set of simple roots.
    \item $\Phi := \{w\alpha_s \mid w \in W, s\in S\}$ the root system.
    \item Let $\alpha \in \Phi$, then it decomposes as
    \begin{align*}
        \alpha = \sum_{s\in S} \lambda_s \alpha_s.
    \end{align*}
    If $\lambda_s \geq 0$ $\forall s \in S$, then $\alpha$ is a \textit{positive} root, and if $\lambda_s \leq 0$ $\forall s \in S$, then it is a \textit{negative} root. We denote by $\Phi^+$ the set of positive roots and by $\Phi^-$ the set of negative roots.
    \item $N(w) := \{\alpha \in \Phi^+ \mid w \alpha \in \Phi^- \}$. 
\end{enumerate}
\end{definition}


We now collect some useful results on these objects; they are a compilation of \cite[Proposition 1.1]{brinkhowlett}, \cite[Proposition 1.2]{ParkYau} and \cite[Proposition 4.4.5]{bjorner_combinatorics_2005}.

\begin{proposition} \label{prop:root-properties}
    Let $(W,S)$ be a Coxeter system. Then, 
 \begin{enumerate}
    \item $\Phi = \Phi^+ \sqcup \Phi^-$. 
    \item $\lvert N(w) \rvert = \ell(w)$, where $\ell(w):= \min \{n \in \NN \mid w = s_1 \ldots s_n, s_i \in S\}$.  
    \item Let $w \in W$ and $s\in S$, then $\ell(ws) > \ell(w) \iff w \alpha_s \in \Phi^+$.
    \item \label{prop:bb-bijection}
    The mapping
    \begin{align*}
        \Phi^+ & \to R \\
        \gamma= w\alpha_s & \mapsto t_\gamma= wsw^{-1}
    \end{align*}
    is a bijective correspondence between positive roots and reflections.
    \end{enumerate}
\end{proposition}

\begin{definition}[\cite{brinkhowlett}, Definition 2.1] 
\label{def:root-domination}
    Let $\alpha, \beta \in \Phi^+$. $\alpha$ \textit{dominates} $\beta$ (with respect to $W$), denoted $\alpha \dom \beta$ if for all $w \in W, \alpha \in N(w) \implies \beta \in N(w)$.
\end{definition}

For convenience, we record some equivalent rewritings of this dominance relation.

\begin{lemma} \label{lemma:root-dom-equivs}
Let $\alpha, \beta \in \Phi^+$. Then,
\begin{align}
    \alpha \dom \beta & \iff \forall w \in W, \alpha \in N(w) \implies \beta \in N(w) \nonumber \\
    & \iff \forall w \in W, w\alpha \in \Phi^- \implies w\beta \in \Phi^- \nonumber \\
    & \iff \forall v \in W, v^{-1}\alpha \in \Phi^- \implies v^{-1}\beta \in \Phi^-\label{eq:root-dom-vinv}.
\end{align}
\end{lemma}

\begin{proof}
    The first equivalence is the definition of $\dom$, the second is the definition of $N$, the third follows by setting $v:=w^{-1}$.
\end{proof}

\begin{definition} [\cite{ParkYau}, Section 1.3] 
\label{def:root-elementary}
    A root $\alpha \in \Phi^+$ is \textit{elementary} if $\alpha$ dominates no other positive root $\beta \neq \alpha$.
\end{definition}

The study of positive roots and their dominance relation was the key to showing the Parallel Wall Theorem (PWT) for Coxeter groups. Originally stated in a preprint by Davis and Shapiro \cite[Theorem 1.7]{davis_coxeter_1991}, 
it was proved by Brink and Howlett \cite[Theorem 2.8]{brinkhowlett}, who showed that the number of \textit{elementary roots} of a Coxeter system of finite rank is finite; see Section \ref{sec:PWT-dyer} for details.




\subsection{Dyer groups}
The definition of Coxeter groups can be generalised by adding a vertex-labelling map to the Coxeter graph; the resulting groups are Dyer groups.
\begin{definition} [\cite{PaV}, Definition 3.1]
\label{def:dyer}
    Let $\GG = \big( V(\GG), E(\GG) \big)$ be a finite simplicial graph. 
    We endow $E(\GG)$ and $V(\GG)$ with the following edge- and vertex-labelling maps, respectively
    \begin{align*}
        m : \hspace{0.5em} & E(\GG) \to \NN_{\geq 2}, \\
        f : \hspace{0.5em} & V(\GG) \to \NN_{\geq 2} \cup \{\infty\},
    \end{align*}
    with the additional \textit{compatibility assumption} that for each edge $e = \{u,v\} \in E(\GG)$,
    \begin{equation}
    \label{eq:compatibility}
        m(e) \neq 2 \implies f(u) = 2 = f(v). 
    \end{equation}
    The triple $(\GG,m,f)$ is called a \textit{Dyer graph} and the group $D = D(\GG, m, f)$ associated to the data $(\GG,m,f)$ is called a \textit{Dyer group} and is defined by the \textit{standard presentation}
    \begin{equation}
    \label{eq:dyerpres}
        D = \Biggl\langle 
       \begin{array}{l|cl}
                      x_v, v \in V(\GG)  & x_v^{f(v)} = 1 \hspace{0.5em} \forall v \in V(\GG) \text{ such that } f(v) \neq \infty, \\
             & [x_u,x_v]_{m(e)} = [x_v,x_u]_{m(e)} \hspace{0.5em} \forall e = \{u,v\} \in E(\GG)
        \end{array}
     \Biggr\rangle,
    \end{equation}
where $[a,b]_m:= \underbrace{abab \ldots}_{m}$ denotes the alternating product of $a$ and $b$ of length $m$.

\end{definition}

We call the relations of the first kind \textit{order relations}, 
those of the second kind as \textit{dihedral relations}, and the cycles they induce as \textit{dihedral cycles}.

\begin{example}
The class of Dyer groups contains other classes of well-studied groups. Indeed, a Dyer group $D=D(\GG,m,f)$ is
        \begin{itemize}
            \item a Coxeter group, if $f(v) = 2$ for all $v \in V(\GG)$ (see Definition \ref{def:coxeter});
            \item a RAAG, if $f(v) = \infty$ for all $v \in V(\GG)$;
            \item a graph product of cyclic groups, if $m(e) = 2$ for all $e \in E(\GG)$.
        \end{itemize}
Moreover, there exist Dyer groups that do not fall into any of the aforementioned subclasses; see Figure \ref{fig:dyer-gp-example} below for an example.
\begin{figure}[ht]
    \centering
    \input{dyer-gp-example}
    \caption{Example of a Dyer group which is neither a Coxeter group nor a graph product of cyclic groups, for $m \geq 3$.}
    \label{fig:dyer-gp-example}
\end{figure}
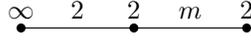
\end{example}

Let $X:= \{x_v \mid v \in V(\GG)\}$; we refer to
$X^{\pm 1} := X \cup X^{-1}$ as the \textit{standard Dyer generating set} of a Dyer group. For each generator $x_v \in X$, we use the shorthand
\begin{equation*}
    \ZZ_{f(v)} := \begin{cases}
        \ZZ/f(v)\ZZ, & \text{if } f(v) < \infty,\\
        \ZZ, & \text{if } f(v) = \infty.
    \end{cases}
\end{equation*}
We refer to the set 
\begin{equation*}
    \langle x_v \rangle = \{x_v^\alpha \mid \alpha \in \ZZ_{f(v)}\}
\end{equation*}
as the \textit{vertex group of $x_v$}. Note that $\langle x_u \rangle \cap \langle x_v \rangle = \{1\}$ for $u\neq v$.

\begin{definition}[\cite{PaS}, Section 3] \label{def:reflections}
The \textit{set of reflections} of a Dyer group is the set 
 \begin{align*}
        R:=\{gxg^{-1} \mid g \in D, x \in X\}.
    \end{align*}
\end{definition}

It will be convenient to consider two other generating sets. 
\begin{definition} \label{def:gen-sets}
    Let $X$ be as above. 
    \begin{itemize}
        \item The \textit{set of syllables} of $X$ \cite[Section 2]{PaS} is
        \begin{align*}
            S(X):=\{ x_v^{\alpha} \mid x_v\in X, \alpha \in \ZZ_{f(v)} \setminus \{0\} \}.
        \end{align*}

        \item $\begin{aligned}[t]
            X_S := \{x_v^{\alpha} \mid x_v\in X \text{ s.t. } f(v) < \infty, \alpha & \in \ZZ_{f(v)} \setminus \{0\}\} \\
                & \cup \{x_v^{\pm 1} \mid x_v\in X \text{ s.t. } f(v) = \infty\}.
        \end{aligned}$
    \end{itemize}
\end{definition}

Clearly, both $S(X)$ and $X_S$ also generate the group $D$. It is important to note that because the vertex-labelling map $f$ may take the value $\infty$, $S(X)$ is \textit{not} necessarily finite; this is remedied by using $X_S$. Observe that 
all three sets coincide for Coxeter groups and that $X_S=X^{\pm 1}$ for RAAGs, since every generator is respectively an involution or of infinite order.\\

For a fixed generating set $A$ of a Dyer group $D$, 
it is natural to consider the associated Cayley graph $\cayda$, whose relevant graph-theoretic properties we review in the following subsection.


\subsection{Mediangle graphs and hyperplanes}
\label{subsec:graphs}
To set the notation, let $\Gamma=(V,E)$ be a connected graph with vertex set $V$ and edge set $E$. 
Recall that a \textit{path} in $\Gamma$ is a sequence of vertices such that any two consecutive vertices are adjacent, and that the \textit{length} of a path is equal to the number of edges it contains.

\begin{definition} 
    Let $x,y \in V$, the \textit{interval} between $x$ and $y$ is
    \begin{equation*}
        I(x,y) := \{z \in V \mid d(x,y) = d(x,z) + d(z,y)\},
    \end{equation*}
    where the distance between two points is equal to the length of a shortest path joining them.
    An interval thus corresponds to the union of all the geodesics between these two vertices.
\end{definition}

\begin{definition} 
    A subgraph $\Upsilon$  of $\Gamma$ is said to be:
    \begin{itemize}
        \item \textit{isometric} if $d_{\Upsilon}(x,y) = d_{\Gamma}(x,y) \hspace{0.5em} \forall~x,y~\in ~V(\Upsilon)$. 
        \item \textit{convex} if $I(x,y) \subseteq \Upsilon \hspace{0.5em} \forall~x,y~\in ~V(\Upsilon)$.
        \item \textit{gated} if for every $x\in V$, there exists a unique $y \in V(\Upsilon)$ such that $y \in I(x,z)$ for all $z \in V(\Upsilon)$. The vertex $y$ is called the \textit{gate} of $x$ and denoted $\pi(x)$.
    \end{itemize}

\end{definition}

\begin{lemma} \label{lemma:gated-convex-isometric}
    Let $\Gamma = (V,E)$ be a graph and let $\Upsilon$ be an induced subgraph. Then, 
    \begin{align*}
        \Upsilon \text{ gated } \implies \Upsilon \text{ convex } \implies \Upsilon \text{ isometric}.
    \end{align*}
    
\end{lemma}
\begin{proof}
    Let $\Upsilon$ be the subgraph of $\Gamma$ induced by $Y\subseteq V$. For the first implication, let $x,y \in Y$ and consider a geodesic $\gamma = \{x_0=x, x_1, \ldots, x_n=y\}$ in $\Gamma$ between them. We show by induction that $x_i \in Y$ for all $i$. Assume for contradiction that $x_1 \notin Y$. Since $\Upsilon$ is gated, the gate of $x_1, \pi(x_1)$, satisfies $\pi(x_1) \in I(x_1, z)$ for every $z \in Y$. Take $z=x_0=x \in Y$, then  $\pi(x_1) \in I(x_1, x)$ implies that $\pi(x_1) = x$ since $I(x_1, x)=\{x_1,x\}$. Now take $z=y: x \in I(x_1,y)$, which contradicts that $\gamma$ is a geodesic path between $x$ and $y$. Suppose that $x, x_1, \ldots, x_{i-1} \in Y$ and assume for contradiction that $x_i \notin Y$. Since $\Upsilon$ is gated, there is a unique gate of $x_i$ such that $\pi(x_i) \in I(x_i, z)$ for every $z \in Y$. Similarly to $x_1$, taking $z=x_{i-1}$ shows that $\pi(x_i) = x_{i-1}$, and taking $z=y$ leads to a contradiction. \\
    For the second implication, any path of $\Upsilon$ is also a path in $\Gamma$, so for all $x,y \in Y$,
    \begin{align*}
        d_{\Gamma}(x,y) \leq d_{\Upsilon}(x,y).
    \end{align*}
    On the other hand, let $x,y \in Y$ and consider a geodesic $\gamma$ in $\Gamma$ between them. Since $\Upsilon$ is convex, $\gamma$ lies entirely in $\Upsilon$, hence
    \begin{align*}
        d_{\Upsilon}(x,y) \leq \ell(\gamma) = d_{\Gamma}(x,y). 
    \end{align*}
\end{proof}


\begin{example} \label{ex:coxeter-gated}
We detail an important example of a convex subgraph. Let $D=D(\GG,m,f)$ be a Dyer group and consider
\begin{align*}
    V_2:=\{v \in V(\GG) \mid f(v)=2\} \subseteq V(\GG)
\end{align*}
and the corresponding generators
\begin{align*}
   X_2 := \{x_v \in X \mid f(v)=2\}.
\end{align*}
Denote by $C = \langle X_2 \rangle$ the subgroup of $D$ generated by $X_2$; it is isomorphic to the Dyer group associated with the subgraph of $\GG$ induced by $V_2$ \cite[Section 1]{soergel_dyer_2024}. This group $C$ is a Coxeter group and we denote this system $(C,X_2)$. Its Cayley graph $\cy(C,X_2)$ coincides with the subgraph of $\caydsx$ induced by $C$; this is shown for instance by proving that every edge in $\caydsx$ between two vertices in $C$ is an edge of $(C,X_2)$, using the projection of \eqref{eq:rpjection-piv} below.

It follows from \cite[Lemma 3.5]{Genevois2} (see also \cite[Corollary 6.6]{genevois}) that $\cy(C,X_2)$ is a gated subgraph of $\caydsx$. Moreover, by Lemma \ref{lemma:gated-convex-isometric}
\begin{align*}
    d(x,y)=d_C(x,y) \quad \forall x,y \in C,
\end{align*}
where $d$ denotes the standard graph metric of $\caydsx$, and $d_C$ that of $\cy(C,X_2)$.
\end{example}

We now define a particular class of graphs: that of \textit{mediangle graphs}, as introduced in \cite{genevois}. It contains median graphs and quasi-median graphs as subclasses.

\begin{definition} 
\label{def:mediangle}
    A connected graph $\Gamma = (V,E)$ is said to be \textit{mediangle} if it satisfies the following four conditions:
    \begin{enumerate}
        \item \textbf{Triangle Condition:} For all vertices $o, x,y \in V$ satisfying $d(o,x) = d(o,y)$ and $d(x,y)=1$, there exists a common neighbour $z \in V$ of $x$ and $y$ such that $z \in I(o,x) \cap I(o,y)$.
        
        \item \textbf{Intersection of Triangles:}  $\Gamma$ does not contain an induced copy of $K_{4}^{-}$, that is, the complete graph $K_4$ with one edge removed.
        
        \item \textbf{Cycle Condition:}  For all vertices $o,x,y,z \in V$ satisfying $d(o,x)=d(o,y)=d(o,z)-1$ and $d(x,z)=1=d(y,z)$, there exists a convex cycle of even length that contains the edges $\{x,z\}, \{y,z\}$ and such that the vertex opposite $z$ belongs to $I(o,x) \cap I(o,y)$.
        
        \item \textbf{Intersection of Even Cycles:}  The intersection of any two distinct convex cycles of even lengths contains at most one edge.
    \end{enumerate}
\end{definition}

Genevois' work on periagroups \cite{genevois, Genevois2} shows that this class of graphs describes the geometry of Dyer groups. Indeed:
\begin{theorem}[\cite{genevois}, Theorem 1.1 and Proposition 5.6]
\label{thm:syllables-mediangle}
    Let $D=D(\GG,m,f)$ be a Dyer group, 
    then $\caydsx$ is a mediangle graph.
\end{theorem}






A key structural feature of mediangle graphs are their \textit{hyperplanes}. 
Hyperplanes play an important role in the study of Coxeter groups (see for instance \cite{davis_geometry_2025})
and quasi-median graphs \cite{genevois_groups_2021}, which in particular include the Cayley graphs of graph products of cyclic groups (with respect to their syllabic generating set).


\begin{definition}[\cite{genevois}, p.11]
\label{def:hyperplane}
    Let $\Gamma$ be a mediangle graph. A \textit{hyperplane} is an equivalence class of edges with respect to the transitive closure of the relation that identifies any two edges that belong to a common $3$-cycle or that are opposite in a convex even cycle.  
\end{definition}

In particular, by \cite[Theorem 3.9(iii)]{genevois}, hyperplanes completely characterise geodesics of the graph:
in a mediangle graph, a path is geodesic if and only if it crosses each hyperplane at most once.
Let $\HH$ denote the set of hyperplanes of the graph $\Gamma$; we define and state some useful properties of hyperplanes below. Recall that a clique $Q \subseteq \Gamma$ is a maximal complete subgraph. 
\begin{definition}
Let $H \in \HH$.
\begin{itemize}
    \item The set \textit{$M(H)$} 
    is the set of endpoints of edges in $H$.
    
    \item The \textit{carrier} of a hyperplane $H$, denoted $N(H)$, is the union of the edges in $H$ and of the convex cycles containing edges in $H$.

    \item The \textit{sectors delimited by $H$} are the connected components of the graph obtained from $\Gamma$ by removing the edges of $H$.
    
    \end{itemize}
\end{definition} 

\begin{remark}
    Note that $M(H) \subseteq N(H)$.
\end{remark}

\begin{definition} \label{def:transverse}
Let $H,K \in \HH$ be distinct. $H,K$ are \textit{transverse} if some convex even cycle contains a pair of opposite edges in $H$ and another in $K$.
\end{definition}

\begin{definition} \label{def:separate}
    Let $H,K \in \HH$, $u,v\in V$. 
\begin{enumerate}[label=(\roman*)]
    \item The hyperplane $K$ \textit{separates the vertices $u$ and $v$} if they lie in distinct sectors of $K$.
    \item The hyperplane $K$ \textit{separates the vertex $v$ from the hyperplane $H$} if $v$ and $x$ lie in distinct sectors of $K$ for every $x \in N(H)$.
\end{enumerate}
\end{definition}


Recall that if two transverse hyperplanes $H,K \in \HH$ both cross a convex even cycle $\CC$ of length $2m$, the \textit{angle} between them is $\pi\frac{1+d}{m}$, where $d$ is the distance in $\CC$ between $H$ and $K$, and that this notion does not depend on the choice of $\CC$; see \cite[p.~9]{Genevois2} for more details.

\begin{definition}
A hyperplane $H \in \HH$ is \textit{right} if the angle between $H$ and $J$ is $\frac{\pi}{2}$ for every $J \in \HH$ transverse to $H$.
\end{definition}

In what follows, we take $\Gamma=\caydsx$ and denote by $\HH$ its set of hyperplanes.
Recall that $D$ acts on itself by left multiplication;
these notions are invariant under this action.

\begin{lemma} \label{lemma:D-invariance}
    Let $H,K \in \HH$ and $g,x \in D$. Then,
    \begin{enumerate}[label=(\roman*)] 
        \item \label{lemma:D-invariance-gH-hyp}
            $gH \in \HH$.
        \item \label{lemma:D-invariance-MN}
            $N(gH) = g N(H)$ and $M(gH) = g M(H)$.
        \item \label{lemma:D-invariance-separates}
            $K$ separates $x$ from $H$ if and only if $gK$ separates $gx$ from $gH$.
        \item \label{lemma:D-invariance-transverse}
            $H,K$ are transverse if and only if $gH,gK$ are transverse.
        \item \label{lemma:D-invariance-right} 
            $gH$ is right if and only if $H$ is right.
    \end{enumerate}
\end{lemma}

\begin{proof}
    The results \ref{lemma:D-invariance-gH-hyp}--\ref{lemma:D-invariance-transverse} follow from left multiplication by $g$ being a graph automorphism, which preserves the adjacency relations of the vertices. In particular, $3$-cycles, convex even cycles, opposite edges, cliques and sectors are preserved. 
For \ref{lemma:D-invariance-right}, since angles are defined in terms of convex cycles \cite[p.~9]{Genevois2}, left multiplication by $g$ moreover preserves the angle between two transverse hyperplanes.
\end{proof}

The following result is \cite[Lemma 2.26]{Genevois2}.
\begin{lemma} \label{lemma:right-carrier-gated} 
In a mediangle graph, the carrier of a right hyperplane is gated.
\end{lemma}

We end this subsection by recalling some results about the cliques of $\caydsx$; they are respectively \cite[Lemma 5.7]{genevois} and \cite[Lemma 6.3]{genevois}.

\begin{lemma} \label{lemma:cliques}
    \begin{enumerate}[label=(\roman*)]
        \item \label{lemma:cliques-cosets} Cliques coincide with the cosets of the vertex groups $g \langle x_v \rangle$, for $g \in D, v \in V(\GG)$.

        \item \label{lemma:cliques-same-stab} Any two cliques in $\caydsx$ that lie in the same hyperplane have the same stabiliser. 
    \end{enumerate}
\end{lemma}

\begin{definition} \label{def:rotstab}
    The \textit{rotative stabiliser} of a hyperplane $H \in \HH$, denoted $\stbr(H)$, is the stabiliser of a clique in $H$ with respect to the action by left-multiplication.
\end{definition}

\begin{remark}
    This notion is well defined by Lemma \ref{lemma:cliques}\ref{lemma:cliques-same-stab}.
\end{remark}

\begin{lemma} \label{lemma:free-transitive-action}
    \cite[Corollary 6.4]{genevois}
    For every $H \in \HH$, $\stbr(H)$ is conjugate to a vertex group and it permutes freely and transitively the sectors delimited by $H$.
\end{lemma}


\subsection{Cone types}
\label{subsec:conetypes}
Although cone types do not play an explicit role in this work, we provide their definition here for completeness. 
Let $G$ be a finitely generated group with finite and symmetric generating set $S$, and denote by $\cay$ its Cayley graph. 

\begin{definition} 
\label{def:word-cone-type}
    Let $g\in G$, the \textit{(word)\footnote{There exists a coarser notion of cone types, that of \textit{graph cone types}, studied for instance in \cite{HowNag}.} 
    cone type} of $g$ is
    \begin{equation*}
        T(g):= \{h\in G \mid \ell_S(gh) = \ell_S(g) + \ell_S(h) \},
    \end{equation*}
    where $\ell_S(g):=\min\{n\in \NN \mid g=s_{1}\ldots s_{n}, s_{i} \in S\}.$
\end{definition}
    Intuitively, the cone type of an element consists of the elements that extend it geodesically. For more details, see for instance \cite{Epstein} or \cite{ParkYau}.



\begin{definition}
    The pair $(G,S)$ is said to be a \textit{Cannon pair} if $\cay$ has finitely many cone types.
\end{definition}

It is known that having finitely many cone types implies regularity of the language of geodesics; the converse is true under an additional assumption. Indeed, Neumann and Shapiro showed that it holds if all relators have even length; see \cite[Section $5.6$]{NS-course} and \cite[Lemma $4$]{CET}. \\

As mentioned previously, examples of Cannon pairs can be found for instance in \cite{brinkhowlett}, \cite{grigorchuk_growth_2022} and \cite{LMW}, but we note that there are also pairs which are known \textit{not} to form Cannon pairs, such as the lamplighter groups $L_m$ with the standard wreath product generating set and Thompson's group $F$ with its standard generating set \cite{CET}. Finally, having finitely many cone types is sensitive to the choice of generating set and thus is \textit{not} a quasi-isometric invariant, as shown by the following example, which can be found in \cite[Example $4.4.2$]{Epstein}.\\

\begin{example}
\label{ex:fmct-not-qi}
    Let $G$ be the split extension of $\ZZ^2 = \langle a,b\rangle$ by $\ZZ/2\ZZ=\langle t \rangle$. It admits the presentation
\begin{equation}
\label{eq:ex-fmct}
    \langle a,b,t \mid t^2=1, ab=ba, tat=b \rangle,
\end{equation}
to which we can add the generators $x=b^2$ and $y=ba$ to get
\begin{equation*}
\label{eq:ex-no-fmct}
    \langle a,b,t,x,y \mid t^2=1, ab=ba, tat=b, x=b^2, y=ba\rangle.
\end{equation*}
It can be shown that $(G, \{a,b,t,x,y\})$ does not have a regular language of geodesics, but that $(G, \{a,b,t\})$ does. Since all its relators have even length, the latter is also a Cannon pair.
\end{example}


\subsection{Falsification by fellow-traveller property (FFTP)}
In this subsection, we follow the exposition of \cite{elder1}.
Let $\Gamma = (V,E)$ be a connected graph with vertex set $V$ and edge set $E$. 
Given a finite path $\gamma~=~\{x_0,~\ldots,~x_n\}$ in $\Gamma$, we will regard it as a function $\gamma: \NN \to V$ such that 
\begin{equation*}
       \gamma(i) = \begin{cases}
            x_i, & \text{ for } 0 \leq i \leq n, \\
            x_n, & \text{ for } i > n.
        \end{cases}
    \end{equation*}


\begin{definition} 
    Two paths $\gamma$ and $\delta$ in $\Gamma$ 
    are said to \textit{$k$-fellow travel} if 
    \begin{equation*}
        d \big( \gamma(t),\delta(t) \big) \leq k \quad \forall t \geq 0.
    \end{equation*}
    They are said to \textit{asynchronously $k$-fellow travel} if there is a non-decreasing unbounded continuous function $\phi : \NN \to \NN$
    such that
    \begin{equation*}
        d \big( \gamma(t),\delta(\phi(t)) \big) \leq k \quad \forall t \geq 0.
    \end{equation*}
\end{definition}

\begin{definition} 
\label{def:fftp}
    The graph $\Gamma$ satisfies the \textit{(asynchronous) falsification by fellow-traveller property ((A)FFTP)} if there exists a constant $k$ such that any non-geodesic path in $\Gamma$ is (asynchronously) $k$-fellow-travelled by a shorter path with the same endpoints.
\end{definition}

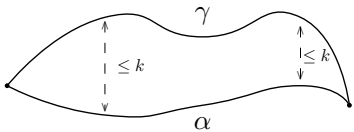
\begin{figure}[ht]
    \centering
    \scalebox{1.3}{\input{fftp}}
    \caption{The paths $\gamma$ and $\delta$ $k$-fellow travel.}
    \label{fig:fftp}
\end{figure} 

We emphasise that the new shorter path $\delta$ need not be geodesic. Elder proved in \cite{elder2} the equivalence between the asynchronous and the synchronous versions of this property; consequently, we will henceforth only consider the FFTP (the \textit{synchronous} version).

\begin{definition}
    Let $G$ be a finitely generated group and $S$ a finite and symmetric generating set. The pair $(G,S)$ is said to \textit{have the FFTP} if $\cay$ does.
\end{definition}
We note that having the FFTP is \textit{not} a quasi-isometric invariant either; it is sensitive to the choice of generating set. Indeed, consider the following example due to Cannon and detailed in \cite{elder1}. 

\begin{example}
\label{example:fftp-not-qi}
Let $G$ be the same virtually abelian group as in Example \ref{ex:fmct-not-qi}. By performing the Tietze transformation of removing the generator $b=tat$ from~\eqref{eq:ex-fmct}, it admits the presentation
\begin{equation*}
\label{eq:ex-no-fftp}
    \langle a,t \mid t^2=1, tata=atat\rangle.
\end{equation*}
It can be shown that $(G, \{a,b,t\})$ has the FFTP, but $(G, \{a,t\})$ does not. 
We note however that the latter does have a regular language of geodesics and finitely many cone types.
\end{example}

We conclude this section by recalling that the FFTP is closely linked to the notions of cone types and regular language of geodesics. Namely, Neumann and Shapiro established the following result; see \cite[Proposition $4.1$]{NS} and \cite[Section $5.6$]{NS-course}.
\begin{theorem} 
\label{thm:NS}
Let $G$ be a finitely generated group with finite symmetric generating set $S$.
If $\cay$ has the FFTP, then $\cay$ has finitely many cone types.
\end{theorem}

We remark that Antolín provided in \cite{Antolin} another proof of Theorem \ref{thm:NS},
by explicitly constructing the corresponding automata. Moreover, it follows from Example \ref{example:fftp-not-qi} that the converse statement does not hold.



\section{The Parallel Wall Theorem for Dyer groups} 
\label{sec:PWT-dyer}

\subsection{Linking hyperplanes and reflections} \label{subsec:hyps-refls}

Recall that for a Dyer group $D=D(\GG,m,f)$, there are notions of \textit{reflections} (Definition \ref{def:reflections}) and \textit{hyperplanes} (Definition \ref{def:hyperplane}); we show that these two notions are in bijective correspondence in $\caydsx$. \\

Let $e= \{g, gx_v^\alpha\}$ be an edge in $\caydsx$, for $g\in D, x_v \in X, \alpha \in \ZZ_{f(v)} \setminus \{0\}$; we associate to it the \textit{reflection data} $\rho(e):=g x_v g^{-1} \in R$. 
Note that it does not depend on the orientation of the edge: $\overline{e} = \{gx_v^\alpha, gx_v^\alpha x_v^{-\alpha}\}$ gives 
\begin{align*}
    \rho(\overline{e}) = gx_v^\alpha x_v (gx_v^\alpha)^{-1} = g x_v g^{-1} = \rho(e).
\end{align*}
Since each edge is contained in a unique hyperplane, we now generalise this edge-association to hyperplanes.

\begin{lemma} \label{lemma:rho-h-well-def}
    The reflection associated to a hyperplane is well defined: it does not depend on the chosen edge.
\end{lemma}

\begin{proof}
    We will show that two edges in the same hyperplane define the same reflection. It suffices to check this for two edges $e,e'$ which either belong to a common $3$-cycle or are opposite edges on a common convex even cycle. 
    \begin{enumerate}[label=Case \arabic*:, ref=\arabic*, leftmargin=*]
        \item Consider $e,e'$ two edges of a common $3$-cycle. By definition, they lie in the same clique, which is by Lemma \ref{lemma:cliques}\ref{lemma:cliques-cosets} the coset of a certain vertex group, $g \langle x_v \rangle$ for some $g \in D, x_v \in X$. Every edge in this clique has the form $\{gx_v^\alpha, g x_v^\alpha x_v^\beta\}$ with $\alpha, \beta \in \ZZ_{f(v)}, \beta \neq 0$, and the associated reflection is $g x_v^\alpha x_v x_v^{-\alpha} g^{-1}= g x_v g^{-1}$. \\
        
        \item Consider $e,e'$ two opposite edges of a convex $2m$-cycle.
        \begin{enumerate}[label=(\roman*)]
            \item If $m=2$, then \begin{align*}
            e=\{g,gx_v^\alpha\}, \quad e'=\{gx_u^\beta, gx_u^\beta x_v^\alpha\},
        \end{align*}
        with $\alpha \in \ZZ_{f(v)} \setminus \{0\}$ and $\beta \in \ZZ_{f(u)} \setminus \{0\}$.
        Then, $\rho(e)=g x_v g^{-1}$ and 
        \begin{align*}
            \rho(e') = gx_u^\beta x_v (gx_u^\beta)^{-1} = gx_u^\beta x_v x_u^{-\beta}g^{-1} = g x_v g^{-1},
        \end{align*}
        since convex even cycles coincide with dihedral cycles by \cite[Claim 4.13]{genevois}, and the $4$-cycle here gives us the relation that the vertex groups generated by $x_v$ and $x_u$ commute.

        \item If $m\geq 3$, the edges of this cycle are labelled alternatingly by the generators $x_v^\alpha$ and $x_u^\beta$. Since convex even cycles coincide with dihedral cycles, we know that there is an edge between $u$ and $v$ in $\GG$ labelled by $m \geq 3$. By the compatibility assumption, we know that $f(u)=f(v)=2$ 
        and moreover, they admit the following relation,
        \begin{align*}
            [x_v,x_u]_m=[x_u,x_v]_m.
        \end{align*}
        Let $e=\{g,gx_v\}$, then $\rho(e)=g x_v g^{-1}$. For its opposite edge $e'$, let $s$ be the alternating word of length $m-1$ beginning in $x_u$ (the other generator to the one labelling $e$), i.e. $[x_u, x_v]_{m-1}$. Then 
        \begin{align*}
            e'=\{gs,gsx_w\}, \text{ where } w=  \begin{cases}
                v, \text{ if } m \text{ is even}, \\
                u, \text{ if } m \text{ is odd}.
            \end{cases}
        \end{align*}
        It follows from the dihedral relation that 
        \begin{align*}
            \rho(e')& =gsx_w(gs)^{-1} = \begin{cases} 
            g \underbrace{x_u \ldots x_u}_{m-1} x_v \underbrace{x_u \ldots x_u}_{m-1} g^{-1}, \text{ if } m \text{ is even} \\
            g \underbrace{x_u \ldots x_v}_{m-1} x_u \underbrace{x_v \ldots x_u}_{m-1} g^{-1}, \text{ if } m \text{ is odd}
            \end{cases} \\
            & = g x_v g^{-1}. \qedhere
        \end{align*}
        \end{enumerate}
    \end{enumerate}
\end{proof}

We denote the reflection of $R$ associated to a hyperplane $H \in \HH$ by $\rho_H$.

\begin{lemma} \label{lemma:D-invariance-reflection}
Let $H \in \HH$ and $g \in D$. Then, $\rho_{gH} = g \rho_H g^{-1}$.
\end{lemma}

\begin{proof}
   Let $e = \{w,wx_v^{\alpha}\}$, with $w \in D, x_v \in X$ and $\alpha \in \ZZ_{f(v)} \setminus \{0\}$, be an edge in $H$; its reflection data is $\rho_H = w x_v w^{-1}$ by Lemma \ref{lemma:rho-h-well-def}. By Lemma \ref{lemma:D-invariance}\ref{lemma:D-invariance-gH-hyp}, $gH$ is a hyperplane. It contains the edge $ge=\{gw,gwx_v^{\alpha}\}$, so by Lemma \ref{lemma:rho-h-well-def},
   \begin{align*}
       \rho_{gH} = gwx_v w^{-1} g^{-1} = g \rho_H g^{-1},
   \end{align*}
   as required.
\end{proof}

\begin{lemma} \label{lemma:stab-H}
    Let $H \in \HH$, then $\stbr(H) = \langle \rho_H \rangle$.
\end{lemma}

\begin{proof}
    Let $Q = g \langle x_v \rangle$ be a clique in $H$, then the edge $\{g,gx_v\}$ lies in $H$ and defines the reflection $gx_v g^{-1}$. By Lemma \ref{lemma:rho-h-well-def}, this reflection corresponds to $\rho_H$. Moreover, recall that by Definition \ref{def:rotstab}, $\stbr(H)$ coincides with the stabiliser of any clique in $H$. The stabiliser of the clique $Q$ is 
    \begin{align*}
        \stb(Q) & = \{ k \in D \mid kQ=Q\} = \{ k \in D \mid k \in g \langle x_v \rangle g^{-1}\} \\
        & = \{ k \in D \mid k \in \langle g x_v g^{-1} \rangle \} = \{ k \in D \mid k \in \langle \rho_H \rangle \}.
    \end{align*}
Hence, $\stbr(H) = \langle \rho_H \rangle$.
\end{proof}

Let us now show that for each $H \in \HH$, the cyclic group $\langle \rho_H \rangle$ contains a unique reflection of $R$.

\begin{lemma} \label{lemma:R-unique}
    Let $H \in \HH$, then $\refh \cap R = \{ \rho_H \}$.
\end{lemma}

\begin{proof}
    Let $H \in \HH$ with $\rho_H=g x_v g^{-1}$ for $g \in D, x_v \in X$. Assume for contradiction that there exists $\rho' \in \refh \cap R$ such that $\rho' \neq \rho_H$. In particular, there exists some power $k \in \ZZ_{f(v)} \setminus \{0\}$ such that $\rho' = \rho_H^k$. 

    \begin{enumerate}[label=Case \arabic*:, ref=\arabic*, leftmargin=*]
        \item $f(v)=2$. It follows that $k=1$, hence $\rho' = \rho_H$. \\

        \item $f(v) \geq 3$, then any edge in $\GG$ emanating from $v$ has label $2$. 
        Consider the projection 
        \begin{align} \label{eq:rpjection-piv}
            \pi_v : D & \twoheadrightarrow \ZZ_{f(v)} = \langle t \rangle \\
            x_w & \mapsto \begin{cases}
                t, \text{ if } w=v,\\
                1, \text{ if } w \neq v;
            \end{cases} \nonumber  
        \end{align}
        it is well-defined as it preserves the relations of $D$.
        %
        Let $\rho'= h x_u h^{-1}$ for $h \in D, x_u\in X$, then $h x_u h^{-1} = \rho_H^k$ implies $\pi_v(x_u) = t^k$.
        Since $\pi_v(x_u) \in \{1,t\}$ and $k \neq 0$, the equality holds if and only if $u=v$ and $k=1$. Hence, $\rho' = \rho_H$. \qedhere
    \end{enumerate}   
\end{proof}

Finally, we establish the following correspondence between hyperplanes in $\caydsx$ and the set of reflections $R$. Recall that we denote by $\rho_H$ the unique reflection in $R$ such that $\stbr(H) = \refh$.

\begin{proposition} \label{prop:bijection-H-R}
    The mapping 
    \begin{align*}
        \HH & \to R \\
        H & \mapsto \rho_H 
    \end{align*}
    is a bijective correspondence between hyperplanes and reflections.
\end{proposition}

\begin{proof}
    This map is well-defined by Lemma \ref{lemma:rho-h-well-def}. 
    
    For surjectivity, take $r \in R$. By definition, $r= g x_v g^{-1}$ for some $g \in D, x_v \in X$. Consider the hyperplane $H \in \HH$ which contains the edge $\{g, g x_v\} \in E(\caydsx)$, then $\rho_H=r$. 
    
    For injectivity, assume $\rho_H=\rho_K=:r\in R$, then $r \in \stbr(H) \cap \stbr(K)$ by Lemma \ref{lemma:stab-H}. Let $Q$ be a clique in $H$, take $x \in V(Q)$ and let $U$ be the sector of $K$ containing $x$. Since $r \in \stbr(K)$, we have $rK=K$, so $r$ permutes the sectors of $K$; in particular, $rx$ lies in the sector $rU$. As $1 \notin R$, $rU\neq U$ since $\stbr(K)$ permutes the sectors of $K$ freely by Lemma \ref{lemma:free-transitive-action}. Since $\stbr(H)$ stabilises $Q$, $rx\in Q$ and $x\neq rx$ since they lie in the distinct sectors $U$ and $rU$, so $\{x,rx\}$ lies in $Q$, hence in $H$. As an edge joining two different sectors must lie in $K$, $\{x,rx\} \in K$, and since every edge belongs to exactly one hyperplane, $H$ and $K$ coincide.  
\end{proof}

\begin{remark}
    The inverse map is given by 
    \begin{align*}
        R & \to \HH \\
        \rho & \mapsto J_\rho,
    \end{align*}
    where $J_\rho$ is the unique hyperplane with 
    $\rho_{J_\rho} = \rho$.
\end{remark}

Since a hyperplane is uniquely paired with a single reflection of $R$, the latter can be used to characterise the hyperplane, in the following sense.

\begin{definition}
    Let $H\in \HH$. $H$ is \textit{thin} if every clique of $H$ is an edge, and it is \textit{thick} otherwise. 
\end{definition}

\begin{lemma} \label{lemma:thin-equivs}
    Let $H \in \HH$, the following are equivalent:
    \begin{enumerate}[label=(\roman*)]
        \item \label{lemma:thin-equivs-1}
        $H$ is thin
        \item \label{lemma:thin-equivs-2}
        $H$ delimits $2$ sectors
        \item \label{lemma:thin-equivs-3}
        $\lvert \stbr(H) \rvert$ = 2
        \item \label{lemma:thin-equivs-4}
        For some (equivalently, for all) $g \langle x_v \rangle$ clique in $H$, $f(v)=2$.
    \end{enumerate}
\end{lemma}

\begin{proof}


Let $Q=g\langle x_v \rangle$ be a clique of $H$,
\begin{align*}
    f(v) = \lvert \langle x_v \rangle \rvert = \lvert Q \rvert = \lvert \{ \text{sectors of } H \} \rvert = \lvert \stbr(H) \lvert,
\end{align*}
where the third equality is due to \cite[Theorem 3.9(i)]{genevois} and the fourth one to Lemma \ref{lemma:free-transitive-action}. Hence, all cliques in $H$ have the same cardinality, which all four conditions force to be equal to $2$.
\end{proof}

\subsection{Reformulation of the dominance relation}
\label{subsec:dominance}

Let $D=D(\GG,m,f)$ be a Dyer group; we introduce the following definitions for $\cy(D,S(X))$. Denote by $\HH$ its set of hyperplanes.

\begin{definition}
    The \textit{depth} of a hyperplane $H \in \HH$ is
    \begin{align*}
        \dpt(H) := d(1, M(H)).
    \end{align*}
\end{definition}

\begin{definition} \label{def:dominance}
    Let $H,K \in \HH$ be two (distinct) hyperplanes, then $H$ \textit{dominates} $K$, denoted $H \succeq K$, if $K$ separates $1$ from $H$.
\end{definition}

\begin{lemma} \label{lemma:transversality}
Let $H,K \in \HH$.
\begin{enumerate}[label=(\roman*)]
    \item \label{lemma:non-trans-single-sector} If $H \neq K$ and $H,K$ are not transverse, then $N(H)$ lies in a single sector of $K$ (and vice versa).
    
    \item \label{lemma:transverse-in-all-sectors} If $H=K$ or $H,K$ are transverse, then every sector of $K$ contains a vertex of $N(H)$ (and vice versa).

    \item \label{lemma:sep-reformulation} Let $g\in D$. Then, $K$ separates $g$ from $H$ if and only if $H\neq K$, $H,K$ not transverse and $g$ does not lie in the sector of $K$ containing $N(H)$.
\end{enumerate}
\end{lemma}

\begin{proof}
\begin{enumerate}[label=(\roman*)]
    \item Let $\CC \subseteq N(H)$ be a convex cycle containing an edge in $H$ and assume for contradiction that there is an edge of $\CC$ in $K$. Then, $\CC$ is either a $3$-cycle or a convex even cycle. In the former case, all the edges of $\CC$ lie in the same clique, which implies $H=K$. In the latter case, the hyperplanes $H$ and $K$ would be transverse. Hence, we deduce that the hyperplane $K$ contains no edge of any convex cycle in $N(H)$. Since $\CC$ is connected, it lies in one sector of $K$. Finally, by definition of a hyperplane, any two edges in $H$ are linked by a chain of convex cycles in $N(H)$, none of which contain an edge of $K$. Hence, they all must lie in the same sector of $K$. Moreover, if $e\in N(H)$ is an edge of $H$ that does not lie on a convex cycle, then since $K \neq H$, its endpoints lie in the same sector of $K$. \\

    \item If $H=K$, it follows from \cite[Theorem 3.9(i)]{genevois}. Suppose that $H\neq K$ are transverse, then there exists a convex even cycle $\CC$ containing one pair of opposite edges in $H$, and another in $K$. In particular, $\CC \subseteq N(H)$.
    \begin{enumerate}[label=Case \arabic*:, ref=\arabic*, leftmargin=*]
        \item If $K$ is a thin hyperplane, then it delimits exactly two sectors by Lemma \ref{lemma:thin-equivs}. Let $\{x,y\}$ be an edge of $\CC$ in $K$, then $x,y \in N(H)$ and by \cite[Claim 3.14]{genevois} $x$ and $y$ are separated by $K$ and thus lie one in each sector. 

        \item If $K$ is a thick hyperplane, then the edge of $K$ of $\CC$ is an edge of a clique with $\geq 3$ vertices and by \cite[Proposition 3.7] {genevois} $\CC$ is a square. The ensuing dihedral relation implies that 
        \begin{align} \label{eq:commuting-reflections}
            \rho_H \rho_K = \rho_K \rho_H.
        \end{align}
        Let $x \in N(H)$, then it lies in some sector of $K$. By Lemma \ref{lemma:stab-H}, $\stbr(K) = \langle \rho_K \rangle$, which by Lemma \ref{lemma:free-transitive-action} permutes the sectors of $K$ freely and transitively, so for each sector of $K$ there exists $i$ such that $\rho_K^{i}$ sends $x$ to that sector. It remains to check that $\rho_K^{i}x \in N(H)$; to do so, we check that $\rho_K H = H$. Indeed, it then follows that $N(H)$ is $\rho_K$ invariant:
        \begin{align*}
            N(H) = N(\rho_K H) = \rho_K N(H),
        \end{align*}
        where the last equality is due to Lemma \ref{lemma:D-invariance}\ref{lemma:D-invariance-MN}. By Lemma \ref{lemma:D-invariance-reflection}, 
        \begin{align*}
            \rho_{\rho_K H} = \rho_K \rho_H \rho_K^{-1} \overset{\eqref{eq:commuting-reflections}}{=} \rho_H,
        \end{align*}
        and we conclude by Proposition \ref{prop:bijection-H-R} that the two hyperplanes coincide, $\rho_K H = H$.\\
    \end{enumerate}

     \item Assume that $K$ separates $g$ from $H$, then the sector of $K$ containing $g$ does not contain any vertex of $N(H)$. By \ref{lemma:transverse-in-all-sectors}, $H\neq K$ and $H,K$ are not transverse, which implies by \ref{lemma:non-trans-single-sector} that $N(H)$ lies in a single sector of $K$. 
     Conversely, assume that $H\neq K$ and $H,K$ are not transverse, then by \ref{lemma:non-trans-single-sector} $N(H)$ lies in a single sector $U$ of $K$. If $g \notin U$, then $g$ lies in a different sector to every vertex of $N(H)$. Hence $K$ separates $g$ from $H$. \qedhere
\end{enumerate}
\end{proof}

\begin{lemma} \label{lemma:dom-not-trans}
    Let $H,K \in \HH$. If $H \succeq K$, then 
            $H,K$ are not transverse.
            
            
\end{lemma}

\begin{proof}
 
    This lemma follows from Definition \ref{def:dominance} and Lemma \ref{lemma:transversality}\ref{lemma:sep-reformulation}.
\end{proof}

\begin{definition}
    A hyperplane $H \in \HH$ is \textit{elementary} if it dominates no other hyperplane $K \neq H$. We denote
    \begin{align*}
        \EE := \{H \in \HH \mid \nexists K \in \HH \text{ such that } H \succeq K\}
    \end{align*}
    the set of elementary hyperplanes.
\end{definition}

In what follows, fix $(W,S)$ to be a Coxeter system; it is in particular a Dyer group. It is well known \cite[Proposition 4.4.5]{bjorner_combinatorics_2005} that the set of positive roots of $W$ is in bijection with its set of reflections. Denote by $\HH_W$ the set of hyperplanes of $\cy(W,S)$; we add to this bijective correspondence the one linking reflections and hyperplanes shown in Subsection \ref{subsec:hyps-refls}. It leads to the following diagram of bijections:
\begin{align*}
    \Phi^+ \longleftrightarrow & R \longleftrightarrow \HH_W \\
    \alpha \longleftrightarrow & t_\alpha \longleftrightarrow \LL_\alpha,
\end{align*}
where $\LL_\alpha := \LL_{t_\alpha}$ denotes the hyperplane associated to the reflection $t_{\alpha}$. \\

Using the definitions introduced above in the setting of $\cy(W,S)$, we prove the equivalence between the dominance of \textit{positive roots} $\alpha, \beta \in \Phi^+$ introduced in \cite{brinkhowlett}, and the dominance of their \textit{corresponding hyperplanes}. 

\begin{lemma} \label{lemma:cox-2-sectors}
   Let $H \in \HH_W$, then $H$ delimits exactly two sectors.
\end{lemma}
\begin{proof}
    We show that $H$ is a thin hyperplane, which implies by Lemma \ref{lemma:thin-equivs} that it delimits two sectors. Let $Q$ be a clique in $H$. By Lemma \ref{lemma:cliques}\ref{lemma:cliques-cosets}, applicable because a Coxeter group is in particular a Dyer group, the clique $Q$ coincides with a vertex group, i.e. $Q = w \langle s \rangle$ for $w\in W, s\in S$. Since every $s \in S$ is an involution, $\langle s \rangle = \{1,s\}$. It follows that $Q= \{w,ws\}$ is an edge. Hence, $H$ is a thin hyperplane.
\end{proof}

To fix the notation, let $S_K^{\pm}$ be the two sectors of the hyperplane $K$, with the convention that $S_K^+$ contains the identity. 

\begin{remark}
    In this context, the notion of separation can be made more precise:
\begin{align*}
    H \succeq K & \iff K \text{ separates } 1 \text{ from } H \\
    & \iff 1 \text{ and } N(H) \text{ lie in distinct sectors } \\
    & \begin{cases}
        1 \in S_K^+ \\
        N(H) \subseteq S_K^-
    \end{cases} \text{ for } S_K^{\pm} \text{ the two sectors of } K.
\end{align*}
\end{remark}

Before showing the equivalence between the two notions of dominance, we need a small technical lemma about the form of edges in a hyperplane $\LL_r$, for $r \in R$.
\begin{lemma} \label{lemma:edges-cox-hyp}
   Let $r \in R$, then $\LL_r = \{ \{w,ws\} \in E(\cy(W,S)) \mid wsw^{-1} = r\}$.
\end{lemma}
\begin{proof}
By Lemma \ref{lemma:rho-h-well-def}, every edge $\{w,ws\}$ in $\LL_r$ has the same reflection, which is $r$ by definition. By Proposition  \ref{prop:bijection-H-R} and the remark following it, the edge $\{w,ws\}$ such that $wsw^{-1}=r$ lies in the hyperplane $J_r$, which is $\LL_r$ by construction.
\end{proof}

\begin{proposition} \label{prop:root-dom-hyp-dom}
Let $(W,S)$ be a Coxeter system, $\alpha, \beta \in \Phi^+$, $\alpha \neq \beta$. Then,
    \begin{align*}
    \alpha \dom \beta \iff \LL_\alpha \succeq \LL_\beta.
\end{align*}
\end{proposition}

\begin{lemma} \label{lemma:hyperplane-dom-subset}
    Let $H,K$ be two distinct hyperplanes of $\cy(W,S)$. Then, 
\begin{align*}
    H \succ K \iff S_H^- \subseteq S_K^-.
\end{align*}    
\end{lemma}

\begin{proof}
Assume that $S_H^- \subseteq S_K^-$. Let $q \in M(H)$; the goal is to show that $N(H) \subseteq S_K^-$.
\begin{enumerate}[label=Case \arabic*:, ref=\arabic*, leftmargin=*]
    \item $q \in S_H^-$, so $q \in S_K^-$ by hypothesis. 

    \item $q \in S_H^+$, then it follows from Lemma \ref{lemma:edges-cox-hyp} that the edge $e=\{q, \rho_H q\}$ is an edge of $H$ for $q^{-1} \rho_H q \in S$, where $\rho_H \in W$ is the reflection associated to the hyperplane $H$. By \cite[Claim 3.14]{genevois}, the two endpoints of $e$ lie in different sectors of $H$; since $q \in S_H^+$, then $\rho_H q \in S_H^-$. Hence, $\rho_H q \in S_K^-$ by assumption. Since $H \neq K$, they have disjoint sets of edges and in particular, $e \notin K$. This implies that the endpoints of $e$ lie in the same sector of $K$ and since $\rho_H q \in S_K^-$, this sector is $S_K^-$. Hence, $q \in S_K^-$. 
\end{enumerate}
In both cases, $1 \in S_K^+$ (by convention) and $M(H) \subseteq S_K^-$. It remains to show that $N(H) \subseteq S_K^-$. Assume for contradiction that $H,K$ are transverse: there is a convex even cycle $\CC$ of length $2m$ containing a pair of opposite edges in $H$ 
and another in $K$. Removing the edges in $H$ splits $\CC$ into two paths of length $m-1$; denote by $x_1,x_2 \in M(H)$ the endpoints of one of them. Consider this path from $x_1$ to $x_2$ along $\CC$; it crosses the hyperplane $K$, exactly once since it is geodesic \cite[Theorem 3.9(iii)]{genevois}. Hence, $x_1,x_2 \in M(H)$ are separated by $K$ and thus lie in distinct sectors of $K$, which contradicts $M(H) \subseteq S_K^-$. By Lemma \ref{lemma:transversality}\ref{lemma:non-trans-single-sector}, $N(H)$ lies in a single sector of $K$, which must be $S_K^-$, as it contains $M(H)$.\\

Conversely, assume that $H \succ K$ and let $z \in S_H^-$. Fix a geodesic $\gamma = \{1=g_0, g_1, \ldots, g_{n-1}, g_n=z\}$ between $1$ and $z$. Since $1 \in S_H^+$ and $z \in S_H^-$ and $\gamma$ is a geodesic path, exactly one edge of $\gamma$ lies in $H$ by \cite[Theorem 3.9(iii)]{genevois}. Let this edge be $\{g_{j-1},g_j\}$, then $g_{j-1}, g_j \in N(H) \subseteq S_K^-$. Since $1 = g_0 \in S_K^+$ and $g_{j-1} \in S_K^-$, we know that $\eval{\gamma}_{[g_0, g_{j-1}]}$ contains exactly one edge of $K$ and $\eval{\gamma}_{[g_j, g_n]}$ does not contain any edge of $K$. So, the vertices $g_j, \ldots, g_n=z$ lie in the same sector of $K$, which is $S_K^-$ since $g_j \in S_K^-$. We conclude that $z \in S_K^-$, hence $S_H^- \subseteq S_K^-$.
\end{proof}

\begin{lemma} \label{lemma:lalpha-neg-sector}
    Let $\alpha \in \Phi^+$, then $S_{\LL_\alpha}^- = \{w \in W \mid w^{-1}\alpha \in \Phi^-\}$.
\end{lemma}

\begin{proof}
    Let $w \in W$, let $w=s_1 \ldots s_n \in W$ be a reduced expression, and denote by $\gamma$ the corresponding geodesic in $\cy(W,S)$. The $i$-th edge of this path is $\{g_{i-1}, g_{i-1}s_i\}$ and its associated reflection is $g_{i-1}s_i g_{i-1}^{-1}=:t_{\beta_i}$, corresponding to the positive root $\beta_i = g_{i-1}\alpha_{s_i}$ (via the bijection of Proposition \ref{prop:root-properties}\eqref{prop:bb-bijection}). By \cite[Proposition 1.2(3)]{ParkYau}, the $n=\ell(w)$ roots $\beta_i$ defined from the $n$ edges of $\gamma$ are exactly the $n$ elements of $N(w^{-1})$. \\
    Finally, we have the following chain of equivalences:
\begin{align*}
    w \in S_{\LL_\alpha}^- & \iff \LL_\alpha \text{ separates } 1\in S_{\LL_\alpha}^+ \text{ from } w \in S_{\LL_\alpha}^- \\
    & \iff \gamma \text{ crosses the hyperplane } \LL_\alpha \text{ in exactly one edge} \\
    & \iff t_\alpha = t_{\beta_i} \text{ for a } \beta_i \in N(w^{-1}) \\
    & \iff \alpha = \beta_i \\
    & \iff \alpha \in N(w^{-1}) \\
    & \iff w^{-1}\alpha \in \Phi^-.
\end{align*}
Hence, $S_{\LL_\alpha}^- = \{w \in W \mid w^{-1}\alpha \in \Phi^-\}$.
\end{proof}

\begin{lemma} \label{lemma:root-dom-subset}
Let $\alpha, \beta \in \Phi^+$. Then, 
   \begin{align*}
       \alpha \dom \beta \iff S_{\LL_\alpha}^- \subseteq S_{\LL_\beta}^-
   \end{align*} 
\end{lemma}

\begin{proof}
    By Lemma \ref{lemma:root-dom-equivs}, this is equivalent to showing
\begin{align*}
    \eqref{eq:root-dom-vinv} \iff S_{\LL_\alpha}^- \subseteq S_{\LL_\beta}^-,
\end{align*}
which follows immediately from Lemma \ref{lemma:lalpha-neg-sector}.
\end{proof}

\begin{proof}[Proof of Proposition \ref{prop:root-dom-hyp-dom}]
    It follows from Lemmas \ref{lemma:hyperplane-dom-subset} and \ref{lemma:root-dom-subset}.
\end{proof}

\begin{corollary} \label{cor:elementary-roots-hyperplanes}
    Let $\alpha \in \Phi^+$ and denote $\LL_\alpha$ its corresponding hyperplane. Then, $\alpha$ is an elementary root $\iff \LL_\alpha$ is an elementary hyperplane.
\end{corollary}

\subsection{The Parallel Wall Theorem} \label{subsec:pwt-dyer}

This subsection is dedicated to the generalisation of the PWT from Coxeter groups to the Dyer group setting; we start by recalling its statement (using our notation).

\begin{theorem}[Parallel Wall Theorem for Coxeter groups] \label{thm:pwt-coxeter}
    Let $(W,S)$ be a finitely generated Coxeter system. There exists a constant $Q = Q(W,S)$ such that for every $w \in W$ and wall $H$ of $\cy(W,S)$, if the distance from $w$ to $H$ is greater than $Q$, then there exists another wall $K$ parallel to $H$ which separates $w$ from $H$.
\end{theorem}

\begin{remark}
    In the original statement, the hyperplanes are called \textit{walls}, and non-transverse hyperplanes are called \textit{parallel}.
\end{remark}

We leverage the dominance relation for hyperplanes introduced in Subsection \ref{subsec:dominance}, together with its equivalence to Brink and Howlett's dominance for positive roots in the Coxeter group setting. We show the following.

\begin{theorem}[Parallel Wall Theorem for Dyer groups, Theorem \ref{thm:intro-pwt}] \label{thm:pwt-dyer}
    Let $D=D(\GG,m,f)$ be a Dyer group. There exists a constant $Q=Q(\GG,m,f)<\infty$ such that for every $g \in D$ and hyperplane $H$ of $\caydsx$ such that $d(g, M(H)) > Q$, there exists a hyperplane which is non-transverse to $H$ and separates $g$ from $H$.
\end{theorem}

\begin{remark}
In the case where $(D,X)$ is a Coxeter system, we retrieve the original PWT for Coxeter groups. 
\end{remark}

The proof of this theorem is divided into two parts, depending on the nature of the hyperplane $H$; we treat the case of \textit{right} hyperplanes in Subsection \ref{subsec:right-hyps}, and that of \textit{non-right} hyperplanes in Subsection \ref{subsec:not-right-hyps}.

\subsubsection{Right hyperplanes} \label{subsec:right-hyps}
Recall that by Lemma \ref{lemma:right-carrier-gated}, the carrier of a right hyperplane is gated. For a gated set $P$, we denote $\pi_P(g)$ the (unique) gate of the element $g$.

\begin{proposition} \label{prop:hyps-g-gate}
    Let $H\in \HH$ be a right hyperplane. Let $g \notin N(H)$, then any hyperplane defined by an edge of any geodesic between $g$ and $\pi_{N(H)}(g)$ separates $g$ from $H$.
\end{proposition}

\begin{proof}
    By \cite[Lemma 3.17]{genevois}, the hyperplanes separating $g$ from $\pi_{N(H)}(g)$ are exactly the ones separating $g$ from $N(H)$. This is the definition of a hyperplane separating $g$ from $H$.
\end{proof}

\begin{corollary} \label{cor:PWT-right-hyps}
    Every right hyperplane satisfies the Parallel Wall Theorem (Theorem \ref{thm:pwt-dyer}) with constant $Q=0$.
\end{corollary}

\begin{proof}
    Let $H \in \HH$ be a right hyperplane, then any convex even cycle of $N(H)$ is necessarily a $4$-cycle (see \cite[Proposition 2.18]{Genevois2}), so the vertex set of $N(H)$ is $M(H)$. Let $g \in D$ be such that $d(g,M(H))>0$, then $g$ does not lie in the gated subset $N(H)$, which implies there is a path of length at least one between $g$ and its gate $\pi_{N(H)}(g)$ in $N(H)$. By Proposition \ref{prop:hyps-g-gate}, every hyperplane $K$ defined by an edge of this path separates $g$ from $H$. By Lemma \ref{lemma:transversality}\ref{lemma:sep-reformulation}, this implies that $H,K$ are not transverse.
\end{proof}

\subsubsection{Hyperplanes which are not right}
\label{subsec:not-right-hyps}

In this case, it is convenient to consider an equivalent reformulation of the Parallel Wall Theorem (Theorem \ref{thm:pwt-dyer}).

\begin{proposition} \label{prop:pwt-equiv}
Let $D=D(\GG,m,f)$ be a Dyer group. For the same constant $Q=Q(\GG,m,f)$, the following are equivalent:
    \begin{enumerate}[label=(\roman*)]
        \item \label{prop:pwt-equiv-part-1} 
        For every $g\in D$ and hyperplane $H \in \HH$ such that $d(g, M(H))>Q$, there is a hyperplane which is non-transverse to $H$ and separates $g$ from $H$;
        \item \label{prop:pwt-equiv-part-2} 
        Every $H \in \EE$ satisfies $d(1,M(H)) \leq Q$.
    \end{enumerate}
\end{proposition}

\begin{proof}
    Let $g \in D, H \in \HH$ such that $d(g, M(H)) > Q$. Consider $K:=g^{-1}H$. 
    By Lemma \ref{lemma:D-invariance}\ref{lemma:D-invariance-MN},  
    \begin{align*}
        d(1,M(K)) = d(1, M(g^{-1}H)) = d(g, M(H)) > Q.
    \end{align*}
    By hypothesis \ref{prop:pwt-equiv-part-2}, 
    $K \notin \EE$, hence there exists a hyperplane $L \in \HH$ such that $K \succeq L$, i.e. $L$ separates $1$ from $K=g^{-1}H$. Equivalently, the hyperplane $gL$ separates $g$ from $gK=H$. Moreover, the hyperplanes $H$ and $gL$ are non-transverse by Lemmas \ref{lemma:dom-not-trans} and \ref{lemma:D-invariance}\ref{lemma:D-invariance-transverse}.\\
    Conversely, 
    let $H \in \EE$, then by definition, no hyperplane separates $1$ from $H$. By hypothesis \ref{prop:pwt-equiv-part-1}, 
    this implies $d(1, M(H)) \leq Q$.
\end{proof}

This proposition can be refined to subclasses of hyperplanes which are invariant under left-multiplication by $D$; we denote such a subclass of hyperplanes of $\caydsx$ by $\HH^{\inv}$.

\begin{corollary} \label{cor:pwt-equiv-inv}
Let $D=D(\GG,m,f)$ be a Dyer group. For the same constant $Q=Q(\GG,m,f)$, the following are equivalent:
    \begin{enumerate}[label=(\roman*)]
        \item \label{prop:pwt-inv-equiv-part-1} 
        For every $g\in D$ and hyperplane $H \in \HHi$ such that $d(g, M(H))>Q$, there is a hyperplane $K$ non-transverse to $H$ and separating $g$ from $H$;
        \item \label{prop:pwt-inv-equiv-part-2} 
        Every $K \in \EE \cap \HHi$ satisfies $d(1,M(K)) \leq Q$.
    \end{enumerate}
\end{corollary}

\begin{proof}
    The proof is analogous to that of Proposition \ref{prop:pwt-equiv}, with the additional fact that $g^{-1}H \in \HHi$ for all $g \in D, H \in \HHi$.
\end{proof}

In what follows, we will specifically consider $\HHi$ to be the subclass of non-right hyperplanes of $\caydsx$; see Lemma \ref{lemma:D-invariance}\ref{lemma:D-invariance-right}.

\begin{lemma} \label{lemma:not-right-implies-thin}
    Let $H \in \HH$ be a hyperplane. If $H$ is not right, then $H$ is thin.
\end{lemma}

\begin{proof}
  By \cite[Lemma 2.24]{Genevois2}, $H$ delimits two sectors, which is equivalent to $H$ being a thin hyperplane by Lemma \ref{lemma:thin-equivs}.
\end{proof}

Recall from Example \ref{ex:coxeter-gated} that the subset of standard generators $X_2$ of order $2$ of $D$ generates a Coxeter group $(C,X_2)$, such that $\caycx$ is gated, hence convex and isometrically embedded in $\caydsx$ by Lemma \ref{lemma:gated-convex-isometric}.

\begin{definition}
    Let $H \in \HH$. We denote
    \begin{align*}
        H^C & := H \cap E(\caycx), \\
        M_C(H^C) & := M(H)\cap C, \\
        \dpt_C(H^C) & := d_C(1,M_C(H^C)).
    \end{align*}
\end{definition}

\begin{lemma} \label{lemma:HC}
    Let $L$ be a hyperplane of $\cy(C,X_2)$ and $H \in \HH$. Then, 
    \begin{enumerate}[label=(\roman*)]
        \item \label{lemma:hyp-C-in-unique-D} $L$ is contained in a unique hyperplane of $\caydsx$, denoted $L^D$, and $\rho_{L^D} = \rho_L$.
        \item \label{lemma:HC-hyperplane} If $H^C \neq \emptyset$, then $H^C$ is a hyperplane of $\cy(C,X_2)$ with $\rho_{H^C} = \rho_H$.
        \item \label{lemma:HC-MC} If $H^C \neq \emptyset$, then $M_C(H^C)$ is the $M$-set of the hyperplane $H^C$ in $\cy(C,X_2)$.
        \item \label{lemma:HC-from-D} $(L^D)^C = L$. 
    \end{enumerate}
\end{lemma}

\begin{proof}
\begin{enumerate}[label=(\roman*)]
    \item Since $\cy(C,X_2)$ is a convex subgraph of $\caydsx$ (Example \ref{ex:coxeter-gated}), its convex even cycles are also convex even cycles in $\caydsx$. Since there are no $3$-cycles, the edges in $L$ all lie in one hyperplane of $\caydsx$; we denote it $L^D$. Let $e=\{w,wx_v\} \in L$ ($w \in C, v \in V_2)$, then $e \in L^D$ and the reflection data of $e$ is the same for both graphs: $\rho_L = \rho_{L^D}$. \\
    
    \item Let $e=\{w,wx_v\}$ be an edge of $H^C := H \cap E(\caycx)$, with reflection $\rho(e) = wx_vw^{-1}$, with $w \in C$ and $v\in V_2$. Since $e \in H, \rho(e)=\rho_H$ and since $e \in E(\caycx)$, $\rho(e)=\rho_L$ for a hyperplane $L$ of $\caycx$ containing $e$; we show that $H^C = L$. By Proposition \ref{prop:bijection-H-R} applied to $\cy(C,X_2)$, $\rho_H = \rho_L$ and $H^C \subseteq L$. Since $e \in L \subseteq L^D$ and $e\in H$, $L^D=H$, hence $L \subseteq H \cap E(\cy(C,X_2))=H^C$. \\

    \item Let $\{w,wx_v\}$ be an edge of $H^C$, then $v\in V_2$ and $H^C$, hence $H$, is thin by Lemma \ref{lemma:thin-equivs}.  Let $p \in M(H)\cap C$. As $H$ is thin, any edge $\{p,px_v\}$ lying in $H$ satisfies $v \in V_2$ by Lemma $\ref{lemma:thin-equivs}$. Then, $px_v \in C, \{p,px_v\} \in H^C$ and $p \in M(H^C)$. Conversely, the endpoints of an edge of $H^C \subseteq H$ lie in both $C$ and $M(H)$. \\
    
    \item We have $L \neq \emptyset$ and $L \subseteq (L^D)^C$. Moreover, by \ref{lemma:HC-hyperplane}, $(L^D)^C$ is a hyperplane of $\cy(C,X_2)$ containing $L$; they must be equal. \qedhere
\end{enumerate}
\end{proof}


\begin{lemma}  
\label{lemma:elementary-3-equivs} 
Let $H \in \HH$. Then,
    \begin{enumerate}[label=(\roman*)]
        \item \label{lemma:elementary-3-equivs-1} $\rho_H \in C \iff H^C \neq \emptyset$.
        \item \label{lemma:elementary-3-equivs-2} $\rho_H \in C \implies H$ thin.
    \end{enumerate}
\end{lemma}


\begin{proof}
\begin{enumerate}[label=(\roman*)]
    \item 
    Let $\rho_H = g x_v g^{-1}$ and assume that $\rho_H \in C$.  Consider the projection onto $C$
        \begin{align} \label{eq:proj-onto-C}
            \pi_C : D & \to C \\
            x_v & \mapsto \begin{cases}
                x_v, v \in V_2, \\
                1, v \notin V_2.
            \end{cases} \nonumber
        \end{align}
It is well-defined and restricts to the identity on $C$. We get
    \begin{align*}
        \rho_H = \pi_C(\rho_H) = \pi_C(g)\pi_C(x_v)\pi_C(g)^{-1}
    \end{align*}
    and since $\rho_H \neq 1, v \in V_2$. Let $w:= \pi_C(g)$, then $\rho_H = w x_v w^{-1}$, with $w \in C$ and $v \in V_2$. Consider the edge $e = \{w, wx_v\}$, it is an edge of $\caycx$ as both its endpoints lie in $C$. The reflection associated to it is $w x_v w^{-1} = \rho_H$, so $e \in H$ by Proposition \ref{prop:bijection-H-R}. Hence, $e\in H \cap E(\caycx) = H^C$ and $H^C \neq \emptyset$.

    Conversely, assume that $H^C \neq \emptyset$, then there exists an edge $e= \{w,ws\} \in H^C = H \cap E(\caycx)$. In particular, $w \in C$ and $s=x_v \in X_2$. The reflection attached to it is $\rho(e)=w x_v w^{-1} \in C$. Since $e\in H, \rho(e) = \rho_H$.

    \item Assume that $\rho_H \in C$, then $\rho_H = w x_v w^{-1}$ for some $w\in C$ and $v \in V_2$. Then, since $v \in V_2$,
    \begin{align*}
        \lvert \stbr(H) \rvert = \lvert \langle w x_v w^{-1} \rangle \rvert = 2.
    \end{align*}
    By Lemma \ref{lemma:thin-equivs}, $H$ is a thin hyperplane. \qedhere
\end{enumerate}
\end{proof}

\begin{lemma} \label{lemma:thin-implies-inC}
    Let $H \in \EE$, then $H$ thin $\implies \rho_H \in C$.
\end{lemma}
 
\begin{proof}
    Consider $p \in M(H)$ such that $d(1, p) = \dpt(H)=:n$. Then, there exists an $x \in X$ such that $px \in M(H)$ and $\{p,px\}$ is an edge of $H$. The associated reflection is $\rho_H = pxp^{-1}$ and since $H$ is a thin hyperplane, $x \in X_2$ by Lemma $\ref{lemma:thin-equivs}$. We now show that $p \in C$, which implies that $\rho_H \in C$. Let 
    \begin{align*}
        \gamma=\{1=g_0, g_1, \ldots, g_n=p\}
    \end{align*}
    be a geodesic path between $1$ and $p$. Each edge has the form $\{g_{j-1}, g_{j-1}s_j=g_j\}$ with $s_j \in \langle x_{v_j} \rangle$ for $x_{v_j} \in X$; we denote its hyperplane $W_j$ and reflection $\rho_{W_j} = g_{j-1}x_{v_j} g_{j-1}^{-1}$. Note that $H\neq W_j$, for otherwise $g_{j-1}\in M(H)$ and $d(1,g_{j-1})< n=\dpt(H)$. Since $H$ is elementary, $H$ and $W_j$ are transverse by Lemma \ref{lemma:transversality}\ref{lemma:sep-reformulation}: there is a convex even cycle $\CC$ which contains a pair of opposite edges in $H$, and another in $W_j$. 
    Assume for contradiction that $f(v_j) \neq 2$. In this case, the edge of $W_j$ lies in a clique with $\geq 3$ vertices and by \cite[Proposition 3.7] {genevois} $\CC$ is a square, and 
    \begin{align}
        \rho_H \rho_{W_j} & = \rho_{W_j} \rho_H \nonumber \\
        \iff \rho_H g_{j-1}x_{v_j} g_{j-1}^{-1} & = g_{j-1}x_{v_j} g_{j-1}^{-1} \rho_H \nonumber \\
        \implies \rho_H g_{j-1}s_j g_{j-1}^{-1} & = g_{j-1}s_j g_{j-1}^{-1} \rho_H \nonumber \\
        \iff g_{j-1}^{-1} \rho_H g_{j-1} & = s_j^{-1} g_{j-1}^{-1} \rho_H  g_{j-1} s_j. \label{eq:rhoH-rewrite-Wj}
    \end{align}
    Consider the reflection $\rho':=g_{j-1}^{-1} \rho_H g_{j-1}$. By Lemmas \ref{lemma:D-invariance}\ref{lemma:D-invariance-gH-hyp} and \ref{lemma:D-invariance-reflection} and Proposition \ref{prop:bijection-H-R}, it defines the hyperplane $g_{j-1}^{-1}H$ and
    \begin{align*}
        \dpt(g_{j-1}^{-1}H) = d(1, M(g_{j-1}^{-1}H)) = d(1, g_{j-1}^{-1} M(H)) = d(g_{j-1}, M(H)) = n - (j-1),
    \end{align*}
    where the second equality is by Lemma \ref{lemma:D-invariance}\ref{lemma:D-invariance-MN}, and the last one is because $\gamma$ is a geodesic path from $1$ to $p \in M(H)$ realising the depth. \\
    Consider the reflection $\rho'':=g_{j}^{-1} \rho_H g_{j}$; we apply the same reasoning: it defines the hyperplane $g_{j}^{-1}H$ and 
    \begin{align*}
        \dpt(g_{j}^{-1}H) = d(1, M(g_{j}^{-1}H)) = d(1, g_{j}^{-1} M(H)) = d(g_{j}, M(H)) = n - j.
    \end{align*}
    However, by \eqref{eq:rhoH-rewrite-Wj}, 
    \begin{align*}
         \rho''=g_{j}^{-1} \rho_H g_{j} = g_{j-1}^{-1} \rho_H g_{j-1} = \rho',
    \end{align*}
    so these two hyperplanes cannot have different depths.

Hence, for all $j \in  \{1, \ldots, n\}, f(v_j)=2$ and the edges of the geodesic path $\gamma$ are all labelled by generators of order $2$, so $\gamma \subseteq \cy(C,X_2)$ and $p \in C$.
\end{proof}


By Corollary \ref{cor:pwt-equiv-inv}, to show the PWT for hyperplanes which are not right, it is enough to show that $d(1,M(H))$ is uniformly bounded for each elementary thin hyperplane $H$. 
In what follows, let $H\in \EE$ be a thin hyperplane of $\caydsx$. By Lemmas \ref{lemma:thin-implies-inC} and \ref{lemma:elementary-3-equivs}, $H^C \neq \emptyset$, and we show that the properties of $H$ of thinness, depth, transversality and elementarity carry over to the Coxeter counterpart $H^C$ of $H$.

\begin{corollary} \label{cor:cox-preserves}
     Let $H$ be as above. Then,
     \begin{enumerate}[label=(\roman*)]
         \item 
         $\rho_H = \rho_{H^C}$.
         \item $H^C$ is a thin hyperplane.
         \item \label{cor:cox-preserves-depth} $\dpt(H) = \dpt_C(H^C)$.
         
     \end{enumerate}
\end{corollary}

\begin{proof}
\begin{enumerate}[label=(\roman*)]
    \item It follows from Lemma \ref{lemma:HC}\ref{lemma:HC-hyperplane}.
    \item It follows from Lemma \ref{lemma:thin-equivs} since the cliques of $H^C$ are all cliques of $H$.  
    \item On the one hand, 
        \begin{align*}
            \dpt(H) & = d(1,M(H)) \\
            & \leq d(1, M(H) \cap C) = d_C(1, M(H) \cap C) = \dpt_C(H^C).
        \end{align*}
    On the other hand, let $\gamma$ be a geodesic between $1$ and $M(H)$ of length $\dpt(H)$ and denote its endpoint by $p \in M(H)$. It follows from the proof of Lemma \ref{lemma:thin-implies-inC} that $p\in C$ and $\gamma$ lies entirely in $C$. Then,
    \begin{align*}
        \dpt(H) = d(1,p) = d_C(1,p) \geq d_C(1, M(H)\cap C) = \dpt_C(H^C),
    \end{align*}
    where the second equality follows from Lemma \ref{lemma:gated-convex-isometric} as $\cy(C,X_2)$ is an isometric subgraph of $\caydsx$. \qedhere
\end{enumerate}
\end{proof}


\begin{lemma} \label{lemma:cox-preserves-transverse}
    Let $H,K \in \HH$ be such that $H^C,K^C\neq \emptyset$. If $H,K$ are transverse in $\caydsx$, then $H^C, K^C$ are transverse in $\cy(C,X_2)$. 
\end{lemma}

\begin{proof}
    Assume that $H,K$ are transverse hyperplanes in $\caydsx$, then there exists a convex even cycle $\CC$ in $\caydsx$ containing one pair of opposite edges in $H$, and another in $K$. By \cite[Claim 4.13]{genevois} this cycle is dihedral and thus has the form 
    \begin{align*}
        \{g, gx_u^\alpha, gx_u^\alpha x_v^\beta, \ldots, gx_v^\beta\},
    \end{align*}
    for some $g \in D, u,v \in V(\GG)$ such that $\{u,v\} \in E(\GG)$ and $\alpha \in \ZZ_f(u) \setminus \{0\}, \beta \in \ZZ_f(v) \setminus \{0\}$.
   To simplify the notation, if $\{g[x_u,x_v]_h, g[x_u,x_v]_{h+1}\}$ defines an edge of $H$, then we will denote 
   \begin{align*}
       \rho_H=g \underbrace{[x_u,x_v]_{h+1}([x_u,x_v]_h)^{-1}}_{=:\tau_H}g^{-1}.
   \end{align*}
    Denote $m:=m(\{u,v\})$. We start by showing that $u,v \in V_2$.
    \begin{enumerate}[label=Case \arabic*:, ref=\arabic*, leftmargin=*]
        \item \label{case-1-proj-C} $m=2$, then the two pairs of opposite edges are exactly the ones defining $H$ and $K$. WLOG, the reflections are 
        \begin{align*}
            \rho_H = gx_v g^{-1} \quad \text{and} \quad \rho_K = gx_u g^{-1},
        \end{align*}
        which both lie in $C$ by assumption that $H^C, K^C \neq \emptyset$, which is equivalent to $\rho_H, \rho_K \in C$ by Lemma \ref{lemma:elementary-3-equivs}\ref{lemma:elementary-3-equivs-1}.
        Consider the projection onto $C$ \eqref{eq:proj-onto-C} and set $q:= \pi_C(g) \in C$. Then, 
        \begin{align*}
          \rho_H = \pi_C(\rho_H)= \pi_C(g)\pi_C(x_v)\pi_C(g)^{-1} = \begin{cases}
            q x_v q^{-1}, v \in V_2, \\
            1, v \notin V_2.
          \end{cases}
        \end{align*}
        Since $\rho_H \neq 1$, we conclude that $v \in V_2$. We proceed in the same way for $\rho_K$ to get 
        $\rho_K=q x_u q^{-1}$ and $u \in V_2$.
    
        \item $m \geq 3$, then $f(u)=f(v)=2$ by the compatibility assumption. Note that in this case, we also get 
        \begin{align*}
            \rho_H & = \pi_C(\rho_H)=q \tau_H q^{-1}, \\
            \rho_K & = \pi_C(\rho_K)=q \tau_K q^{-1},
        \end{align*}
for $q$ as in Case \ref{case-1-proj-C}
    \end{enumerate}
Now, consider the cycle 
\begin{align*}
    \CC_0 = \{q, qx_u, qx_u x_v, \ldots, qx_v \},
\end{align*}
which lies in $\caycx$ since $q \in C$ and $u,v \in V_2$. This cycle is a convex $2m$-cycle in $\caycx$ by \cite[Claim 4.13]{genevois}, and since $\caycx$ is a convex subgraph of $\caydsx$ (see Example \ref{ex:coxeter-gated}), it is also a convex $2m$-cycle in $\caydsx$.

It remains to check that two pairs of opposite edges of $\CC_0$ belong to $H^C, K^C$ respectively. By construction, the edge $e = \{q [x_u,x_v]_h, q[x_u,x_v]_{h+1}\}$ lies on $\CC_0$. By Lemma \ref{lemma:rho-h-well-def}, $e$ and the edge opposite $e$ on $\CC_0$ define the same reflection, which is
\begin{align*}
    q[x_u,x_v]_{h+1}([x_u,x_v]_h)^{-1}q^{-1} = q \tau_H q^{-1} = \rho_H.
\end{align*}
Hence, this edge and the one opposite it lie in the hyperplane $H$ by Proposition \ref{prop:bijection-H-R}, thus lie in $H \cap E(\caycx) = H^C$. Proceeding in the same way for $K^C$, we get that $\CC_0$ is a convex even cycle of $\caycx$ containing one pair of opposite edges in $H^C$ and another in $K^C$: the hyperplanes $H^C$ and $K^C$ are transverse in $\caycx$.
\end{proof}

\begin{lemma} \label{lemma:HC-elem}
    If $H \in \EE$ and $H^C \neq \emptyset$, then $H^C$ is elementary in $C$. 
\end{lemma}

\begin{proof}
    Assume for contradiction that $H^C$ is not elementary in $\caycx$: there is a hyperplane $L$ of $\caycx$ such that 
    $L$ separates $1$ from $H^C$. Let $K:=L^D$, then by Lemma \ref{lemma:HC}\ref{lemma:HC-from-D}, $L=(L^D)^C = K^C$ and by Lemma \ref{lemma:dom-not-trans}, $H^C$ and $K^C$ are non-transverse in $\caycx$. By Lemma \ref{lemma:HC}\ref{lemma:HC-hyperplane} $H=(H^C)^D$ and $K=(K^C)^D$, and by the previous Lemma \ref{lemma:cox-preserves-transverse}, $H$ and $K$ are non-transverse in $\caydsx$. As $K^C = L\neq H^C$, all of $N(H)$ lies in a single sector of $K$ by Lemma \ref{lemma:transversality}\ref{lemma:non-trans-single-sector}; call it $U$. We have:
        \begin{align*}
            M_C(H^C) = M(H) \cap C \subseteq M(H) \subseteq N(H) \subseteq U.
        \end{align*}
    Since $H^C\neq \emptyset$, take $q \in M_C(H^C)$ and consider a geodesic path in $\caydsx$ between $1$ and $q$. Since both endpoints of this path are in $C$ and $\caycx$ is a convex subgraph of $\caydsx$ (Example \ref{ex:coxeter-gated}), 
    the path lies entirely in $\caycx$. By Lemma \ref{lemma:HC}\ref{lemma:HC-MC}, $q \in M(H^C)$ and since $K^C$ separates $1$ from $H^C$, the vertices $1$ and $q$ lie in distinct sectors of $K^C$. Hence the path contains an edge of $K^C\subseteq K$ and since it is a geodesic of $\caydsx$, no other edge of $K$ by \cite[Theorem 3.9(iii)]{genevois}. As $q\in U$, we get $1\notin U$, so $1$ and $N(H) \subseteq U$ lie in different sectors of $K$, thus $H \succeq K$ in $\caydsx$, contradicting that $H\in \EE$.   
\end{proof}

Finally, we conclude this section with the proof of the PWT for Dyer groups.

\begin{proof}[Proof of Theorem \ref{thm:pwt-dyer}]
Let $H \in \HH$ be a hyperplane. We distinguish two cases, whether $H$ is right or not.

\begin{enumerate}[label=Case \arabic*:, ref=\arabic*, leftmargin=*]
    \item $H$ is right. This case was treated in Subsection \ref{subsec:right-hyps}, where we showed in Corollary \ref{cor:PWT-right-hyps} that the PWT holds for $Q=0$. \\

    \item $H$ is not right. By \cite[Theorem 2.8]{brinkhowlett}, the number of elementary roots of $(C,X_2)$ is finite, which implies by Corollary \ref{cor:elementary-roots-hyperplanes} that the number of elementary hyperplanes in  $\caycx$ is also finite.
    
    Denote the set of elementary hyperplanes in  $\caycx$ by $\EE^C$ and set
\begin{align*}
    Q_C & := \max \{\dpt_C(L) \mid L 
    \in \EE^C\}.
\end{align*}
Note that $\max \emptyset := 0$.

Finally, let $H \in \EE$ and not right. Then,
\begin{align*}
    & H \text{ is thin, by Lemma \ref{lemma:not-right-implies-thin}}\\
    \implies & H^C \neq \emptyset, \text{ by Lemmas \ref{lemma:thin-implies-inC} and \ref{lemma:elementary-3-equivs}\ref{lemma:elementary-3-equivs-1}} \\
    \implies & H^C \in \EE^C, \text{ by Lemma \ref{lemma:HC-elem}} \\
    \implies & \dpt_C(H^C) \leq Q_C < \infty, \text{ by above} \\
    \iff & \dpt(H) \leq Q_C, \text{ by Corollary \ref{cor:cox-preserves}\ref{cor:cox-preserves-depth}}.
\end{align*}

By Corollary \ref{cor:pwt-equiv-inv}, this proves the PWT for non-right hyperplanes. 
\qedhere
\end{enumerate}
\end{proof}

\begin{remark}
   The Parallel Wall Theorem therefore holds for any hyperplane with $Q:=Q_C$, which depends only on the underlying Coxeter system $(C,X_2)$.
\end{remark}

\section{The FFTP for Dyer groups} \label{sec:fftp}
Let $D= D(\GG,m,f)$ be a Dyer group as in Definition \ref{def:dyer}. This section is dedicated to the proof of our main result:  $(D,X_S)$ has the FFTP, where $X_S$ is the finite generating set defined in Definition \ref{def:gen-sets}. The argument is divided into two steps: we first show that $\cy(D,X_S)$ is isomorphic to the Cayley graph of another Dyer group $\td$, taken with its syllabic generating set (Subsection \ref{subsec:XS}), and then apply the Parallel Wall Theorem established in Section \ref{subsec:pwt-dyer} to $\td$ to prove that $(D,X_S)$ has the FFTP (Subsection \ref{subsec:xs-fftp}). 

\subsection{The finite generating set $X_S$} \label{subsec:XS}

Recall from Definition \ref{def:gen-sets} that the generating set $X_S$ consists of
\begin{align*}
     X_S := \{x_v^{\alpha} \mid x_v\in X \text{ s.t. } f(v) < \infty, \alpha & \in \ZZ_{f(v)} \setminus \{0\}\} \\
        & \cup \{x_v^{\pm 1} \mid x_v\in X \text{ s.t. } f(v) = \infty\}.
\end{align*}

Set $V_\infty := \{v \in V(\GG) \mid f(v) = \infty\}$. We define the following operation on the Dyer graph $(\GG,m,f)$. 
\begin{definition} \label{def:infinite-doubling}
    The \textit{infinite-vertex double} of $(\GG,m,f)$ is the graph $(\tgg, \tm, \tf)$, where every vertex $v \in V_\infty$ is replaced by two non-adjacent vertices $v_0, v_1$ labelled by $2$, and every edge $\{u,v\} \in E(\GG)$ by
        \begin{align*}
            & \{u,v\}, \text{ if } u,v \notin V_\infty, \\
            & \{u, v_0\}, \{u, v_1\} \text{ if } u \notin V_\infty, v\in V_\infty, \\
            & \{u_0, v_0\}, \{u_0, v_1\}, \{u_1, v_0\}, \{u_1, v_1\} \text{ if } u,v \in V_\infty,
        \end{align*}
labelled by $m(\{u,v\})$.
\end{definition}

By construction, $(\tgg, \tm, \tf)$ is a Dyer graph; we denote by $\td$ the Dyer group it defines and $\tX$ its standard generating set. Note that since all the vertex labels of $\tgg$ are finite, its set of syllabic generators $S(\tX)$ is also finite.
For $v \in V_\infty$, consider the map $\eta_v : D \to \ZZ$ defined by 
\begin{align*} 
        x_w & \mapsto \begin{cases}
            1, w=v, \\
            0, w \neq v.
        \end{cases}
\end{align*}
It is a well-defined homomorphism and for $g \in D, \eta_v(g)$ is the exponent sum of $x_v$ in $g$.
We now define a \textit{relabelling} $\lambda$ of the edges of the graph $\cy(D,X_S)$ by elements of $S(\tX)$.

\begin{definition} \label{def:lambda}
    Fix $g \in D$, let $z\in X_S$. The new label of the edge $\{g, gz\} \in E(\cy(D,X_S))$ emanating from $g$ and labelled by $z$ is determined as follows:
\begin{align}
    \lambda_g : X_S & \to S(\tX) \label{eq:lambda} \\
            x_u^\alpha & \mapsto \tx_u^\alpha, u \notin V_\infty, \alpha \in \ZZ_{f(u)} \setminus \{0\} \nonumber \\
            x_v & \mapsto \begin{cases}
                \tx_{v_0}, & \text{ if } v \in V_\infty \text{ and } \eta_v(g) \text{ is even,} \\
                \tx_{v_1}, & \text{ if } v \in V_\infty \text{ and } \eta_v(g) \text{ is odd,} \\
             \end{cases} \nonumber \\
             x_v^{-1} & \mapsto \begin{cases}
                \tx_{v_0}, & \text{ if } v \in V_\infty \text{ and } \eta_v(g) \text{ is odd,} \\
                \tx_{v_1}, & \text{ if } v \in V_\infty \text{ and } \eta_v(g) \text{ is even.} \\
             \end{cases} \nonumber
\end{align}
\end{definition}


\begin{lemma} \label{lemma:lambda}
Fix $g \in D$. The map $\lambda_g$ defined in Definition \ref{def:lambda} is
\begin{enumerate}[label=(\roman*)]
    \item \label{lemma:lambda-1} bijective, i.e. each element of $S(\tX)$ labels exactly one edge emanating from $g$;
    \item \label{lemma:lambda-2} compatible with the orientation: for all $z \in X_S$, 
    \begin{align} \label{eq:inverse-gen}
        \lambda_{gz}(z^{-1}) = \lambda_{g}(z)^{-1}.
    \end{align}
\end{enumerate}
\end{lemma}

\begin{proof}
\begin{enumerate}[label=(\roman*)]
    \item We distinguish two cases.
\begin{enumerate}[label=Case \arabic*:, ref=\arabic*, leftmargin=*]
    \item If $z=x_v^{\pm 1}$ for $v \in V_\infty$, then the corresponding edges of $\cy(D,X_S)$ are $\{g,gx_v\}$ and $\{g,gx_v^{-1}\}$ (distinct since $x_v$ has infinite order). By definition, $\lambda_g(x_v) = \tx_{v_i}$ and $\lambda_g(x_v^{-1}) = \tx_{v_{j}}$, where $j = 1-i \mod 2$, i.e. one of each $\tx_{v_0}$ and $\tx_{v_1}$.
    
    \item If $z=x_u^\alpha$ for $u \notin V_\infty, \alpha \in \ZZ_{f(u)} \setminus \{0\}$, then $\lambda_g(z) = \tx_u^\alpha$. Since $x_u$ and $\tx_u$ both have order $f(u)$, this is a well-defined bijective correspondence between $\langle x_u\rangle \setminus \{1\}$ and $\langle \tx_u\rangle \setminus \{1\}$. \\
\end{enumerate}

    \item We distinguish the same two cases.
\begin{enumerate}[label=Case \arabic*:, ref=\arabic*, leftmargin=*]
    \item If $z=x_v$ for $v \in V_\infty$, then the reversed edge is $\{gx_v, g\}$ labelled by $x_v^{-1}$. By definition of $\lambda$,
    \begin{align*}
        \lambda_{g x_v}(x_v^{-1}) = \tx_{v_i}, \text{ where } i= \begin{cases}
            0, \eta_v(gx_v) \text{ is odd,} \\
            1, \eta_v(gx_v) \text{ is even.}
        \end{cases}
    \end{align*}
    Since $\eta_v(gx_v)= \eta_v(g)+1$, we see that it coincides with $\lambda_g(x_v) = \lambda_g(x_v)^{-1}$, as $\tx_{v_i}$ is an involution. The case $x_v^{-1}$ is analogous.
    
    \item If $z=x_u^\alpha$ for $u \notin V_\infty, \alpha \in \ZZ_{f(u)} \setminus \{0\}$, then the reversed edge $\{g x_u^\alpha, g\}$ is labelled by $x_u^{-\alpha}$. By definition of $\lambda$,
     \begin{align*}
        \lambda_{g x_u^\alpha}(x_u^{-\alpha}) = \tx_{u}^{-\alpha} = \lambda_g(x_u^\alpha)^{-1},
    \end{align*}
    as required. \qedhere
\end{enumerate}
\end{enumerate}
\end{proof}

\begin{remark}
For the next lemma, it is convenient to make explicit the inverse map $\lambda_g^{-1}$ of $\lambda_g$:
\begin{align} \label{eq:lambda-inverse}
    \lambda^{-1}_g : S(\tX) & \to X_S \\  
        \tx_u^\alpha & \mapsto x_u^\alpha, \text{ if } u \notin V_\infty, \alpha \in \ZZ_{f(u)} \setminus \{0\}, \nonumber \\
        \tx_{v_0} & \mapsto \begin{cases}
            x_v, & \eta_v(g) \text{ is even} \nonumber \\
            x_v^{-1}, & \eta_v(g) \text{ is odd}, \nonumber \\
        \end{cases} \\
        \tx_{v_1} & \mapsto \begin{cases}
            x_v, & \eta_v(g) \text{ is odd} \nonumber \\
            x_v^{-1}, & \eta_v(g) \text{ is even}.
        \end{cases}
\end{align}

\end{remark}

For $A \in \{X_S, S(\tX)\}$, denote by $A^*$ the set of words in $A$, i.e. the set of finite sequences of elements of $A$. Recall that each word $w=(a_1, \ldots, a_n)$ represents an element $g$ of the ambient group; we denote this correspondence by $g=\overline{w}=a_1\ldots a_n$. The inverse of a word $w=(a_1, \ldots, a_n)$ is defined as $w^{-1}:=(a_n^{-1}, \ldots, a_1^{-1})$. 

Fix $g\in D$. The map $\lambda_g$ and its inverse $\lambda^{-1}_g$ extend to words in $X_S^*$ and $S(\tX)^*$ by relabelling each letter from the vertex at which it is read: 
\begin{itemize}
    \item let $w = (z_1, \ldots, z_n) \in X_S^*$, then 
\begin{align*}
    \lambda_g(w) := \lambda_g((z_1, \ldots, z_n)) = (\lambda_g(z_1), \lambda_{gz_1}(z_2), \ldots, \lambda_{gz_1 \ldots z_{n-1}}(z_n)).
\end{align*}

    \item let $y = (y_1, \ldots, y_n) \in S(\tX)^*$, then 
\begin{align*}
    \lambda_g^{-1}(y) := (s_1, \ldots, s_n),
\end{align*}
where $s_t := \lambda^{-1}_{gs_1 \ldots s_{t-1}}(y_t)$ for $1 \leq t \leq n$.
\end{itemize}
By definition, $\lambda_g(w_1w_2)=\lambda_g(w_1)\,\lambda_{g\overline{w_1}}(w_2)$ for all $w_1,w_2\in X_S^*$ and iteratively applying \eqref{eq:inverse-gen} yields, for all $w\in X_S^*$,
\begin{align} \label{eq:lambda-inverse-on-words}
    \lambda_g(w^{-1})=(\lambda_{g\overline{w}^{-1}}(w))^{-1}.
\end{align} Moreover, the two maps on words are mutually inverse; by induction on the length of the words, one shows that for all $w\in X_S^*$ and $y\in S(\tX)^*$,
\begin{align*}
    \lambda_g^{-1}(\lambda_g(w)) =w \quad \text{and} \quad \lambda_g(\lambda_g^{-1}(y)) =y.
\end{align*}


\begin{proposition} \label{lemma:1-stays-1}
\begin{enumerate}[label=(\roman*)]
    \item \label{item:closed-word-1-XS} Let $w \in X_S^*$ such that $\overline{w} = 1$ in $D$, then for all $g \in D$,
    \begin{align} \label{eq:closed-word-1-XS}
        \overline{\lambda_{g}(w)} = 1 \text{ in } \td.
    \end{align}

    \item \label{lemma:map-phi-well-def} The map
   \begin{align*}
    \varphi : \cy(D,X_S) & \to \cy(\td, S(\tX)) \\
    gz & \mapsto \varphi(g)\lambda_g(z), \\
\end{align*}
where $g \in D$ and $z \in X_S$, with $\varphi(1)=1$ 
is well defined and surjective.
    \end{enumerate}
\end{proposition}

\begin{proof}
\begin{enumerate}[label=(\roman*)]
    \item Let $w \in X_S^*$ such that $\overline{w} =1$ in $D$. Consider the following set of words
\begin{align*}
    & (z,z^{-1}) \text{ for } z \in X_S, \\
    & (x_u^\alpha, x_u^\beta, x_u^{-\alpha - \beta}) \text{ for } u\notin V_\infty, \alpha, \beta, \alpha + \beta \in \ZZ_{f(u)} \setminus \{0\}, \\
    & \underbrace{(x_u, x_v, \ldots)}_m (\underbrace{(x_v, x_u, \ldots)}_m)^{-1} \text{ for } \{u,v\}\in E(\GG), m=m(\{u,v\});
\end{align*}
they represent $1$ in $D$, and are the building blocks for the relations of $D$. So, the word $w$ can be transformed into the empty word by adding or deleting finitely many of these distinguished words; it thus suffices to check that their image under $\lambda$ represents $1$ in $\td$, for every base vertex. Fix $g \in D$. 
\begin{itemize}
    \item If $w=(z,z^{-1})$ for $z \in X_S$, then
    \begin{align*}
        \lambda_g(w) = (\lambda_g(z), \lambda_{gz}(z^{-1})) \overset{\eqref{eq:inverse-gen}}{=} (\lambda_g(z), \lambda_{g}(z)^{-1}).
    \end{align*}

    \item If $w=(x_u^\alpha, x_u^\beta, x_u^{-\alpha - \beta})$ for $u\notin V_\infty, \alpha, \beta, \alpha + \beta \in \ZZ_{f(u)} \setminus \{0\}$, then
    \begin{align*}
        \lambda_g(w) & = (\lambda_g(x_u^\alpha), \lambda_{gx_u^\alpha}(x_u^\beta), \lambda_{gx_u^\alpha x_u^\beta}(x_u^{-\alpha-\beta})) \\
        & \overset{\eqref{eq:inverse-gen}}{=} (\lambda_g(x_u^\alpha), \lambda_{gx_u^\alpha}(x_u^\beta), \lambda_{g}(x_u^{\alpha+\beta})^{-1}) \\
        & = (\tx_u^\alpha, \tx_u^\beta, (\tx_u^{\alpha+\beta})^{-1}).
    \end{align*}
Hence, $\overline{\lambda_g(w)} = 1$ in $\td$.
    

    \item Let $w_1=\underbrace{(x_u, x_v, \ldots)}_m, w_2 = \underbrace{(x_v, x_u, \ldots)}_m$ where $u,v \notin V_\infty$ such that $\{u,v\} \in E(\GG)$ and $m(\{u,v\})=m$. 
    By Definition \ref{def:lambda}, $\lambda_g(x_u^{\gamma})=\tx_u^{\gamma}$ and $\lambda_g(x_v^{\gamma})=\tx_v^{\gamma}$ for every $g\in D$, so the word $\lambda_g(w)$ is simply obtained from $w$ by adding tildes. By construction, $\{u,v\}$ is also an edge of $\tgg$ labelled by $m$; we get
    \begin{align*}
    \overline{\lambda_g(w)} = \underbrace{\tx_u \tx_v \ldots}_{m} (\underbrace{\tx_v \tx_u \ldots}_{m})^{-1} = [\tx_u,\tx_v]_m ( [\tx_v,\tx_u]_m )^{-1} = 1
    \end{align*}
    in $\td$.

    \item If $w = (x_u, x_v, x_u^{-1}, x_v^{-1})$ for $u \notin V_\infty, v \in V_\infty$, such that $\{u,v\} \in E(\GG)$ (here, $m=2$), then
    \begin{align*}
        \lambda_g(w) & = (\lambda_{g}(x_u), \lambda_{gx_u}(x_v), \lambda_{gx_ux_v}(x_u^{-1}), \lambda_{gx_v}(x_v^{-1})) \\
        & \overset{\eqref{eq:inverse-gen}}{=} (\lambda_{g}(x_u), \lambda_{gx_u}(x_v), \lambda_{gx_ux_v}(x_u^{-1}), \lambda_{g}(x_v)^{-1}) \\
        & = (\tx_u, \tx_{v_i}, \tx_u^{-1}, \tx_{v_j}),
    \end{align*}
    where $i=j$ since $\eta_v(g) = \eta_v(gx_u)$. Hence, $\overline{\lambda_g(w)} = \tx_u \tx_{v_i} \tx_u^{-1} \tx_{v_i} = 1$ as this dihedral relation holds in $\td$.

    \item If $w = (x_u, x_v, x_u^{-1}, x_v^{-1})$ for $u, v \in V_\infty$, such that $\{u,v\} \in E(\GG)$ (here, $m=2$), then similarly to the previous case, we get
    \begin{align*}
        \overline{\lambda_g(w)} = \tx_{u_j} \tx_{v_i} \tx_{u_j} \tx_{v_i} = 1.
    \end{align*}
\end{itemize}

\item Let $w=(z_1,\dots,z_n)\in X_S^*$ be such that $\overline{w}=g$. Applying the definition of $\varphi$ iteratively, the image of $g$ is
\begin{align*}
\lambda_1(z_1)\,\lambda_{z_1}(z_2)\cdots\lambda_{z_1\cdots z_{n-1}}(z_n)
=\overline{\lambda_1(w)}.
\end{align*}
Thus, $\varphi$ is well
defined if $\overline{\lambda_1(w)}$ depends only on $\overline{w}$.
Let $w' \in X_S^*$ be another word such that $\overline{w'}=g$. Then, $\overline{ww'^{-1}} = 1$ and by \ref{item:closed-word-1-XS} and \eqref{eq:lambda-inverse-on-words},
\begin{align*}
    1 & = \overline{\lambda_1(ww'^{-1})} = \overline{\lambda_1(w) \lambda_{g}(w'^{-1})} = \overline{\lambda_1(w)} \, \overline{\lambda_{g}(w'^{-1})} \\
    & = \overline{\lambda_1(w)} \, \overline{(\lambda_{g\overline{w'}^{-1}}(w'))^{-1}} = \overline{\lambda_1(w)} \, \left( \overline{\lambda_{1}(w')} \right)^{-1}.
\end{align*}
Hence, $\overline{\lambda_1(w)} = \overline{\lambda_{1}(w')}$ and $\varphi(g) = \overline{\lambda_1(w)}$ for any $w\in X_S^*$ such that $g=\overline{w}$.

Moreover, the map $\varphi$ is surjective. Let $h=s_1\cdots s_n \in \td$ and denote $h_t:=s_1 \ldots s_t$ for $0 \leq t \leq n$; we proceed by induction on $t$. The base case is 
\begin{align*}
    h_0 :=1 = \varphi(1).
\end{align*}
By induction, there is $g_{t-1} \in D$ such that
\begin{align*}
    h_t=h_{t-1} s_t = \varphi(g_{t-1}) s_t,
\end{align*}
and by Lemma \ref{lemma:lambda}\ref{lemma:lambda-1}, there exists a $z_t \in X_S$ such that $\lambda_{g_{t-1}}(z_t)=s_t$. Hence,
\begin{align*}
    \varphi(g_{t-1}z_t)=\varphi(g_{t-1})\,\lambda_{g_{t-1}}(z_t)=h_{t-1}s_t=h_t.
\end{align*}
\qedhere

\end{enumerate}
\end{proof}



Similarly, for $v \in V_\infty$, consider the map $\teta : \td \to \ZZ/2\ZZ$ defined by
\begin{align*}
        \tx_w & \mapsto \begin{cases}
            1, w=v_0 \text{ or } w=v_1, \\
            0, w\neq v_0, v_1.
        \end{cases}
\end{align*}
It is a well-defined homomorphism and for $h \in \td$, $\teta(h)$ encodes the parity of the number of letters $\tx_{v_i}$ in $h$.
We define a relabelling map $\mu$ of the edges of the graph $\cy(\td, S(\tX))$ by elements of $X_S$. 
\begin{definition} \label{def:mu}
    Fix $h \in \td$, let $s\in S(\tX)$. The new label of the edge $\{h,hs\}$ of $\cy(\td, S(\tX))$ emanating from $h$ and labelled by $s$ is determined as follows:
\begin{align*}
   \mu_h : S(\tX) & \to X_S \\ 
   \tx_u^\alpha & \mapsto x_u^\alpha, u \notin V_\infty, \\
   \tx_{v_0} & \mapsto \begin{cases}
        x_v, & \text{ if } \teta(h) \text{ is even,} \\
        x_v^{-1}, & \text{ if } \teta(h) \text{ is odd,}
        \end{cases} \\
    \tx_{v_1} & \mapsto \begin{cases}
        x_v, & \text{ if } \teta(h) \text{ is odd,} \\
        x_v^{-1}, & \text{ if } \teta(h) \text{ is even.}
        \end{cases}
\end{align*}
\end{definition}

\begin{lemma} \label{lemma:mu-properties}
Let $h \in \td$ and $g \in D$ such that $h=\varphi(g)$; such a $g$ exists by Proposition \ref{lemma:1-stays-1}\ref{lemma:map-phi-well-def}.
    \begin{enumerate}[label=(\roman*)]
        \item \label{lemma:mu-is-lambda-inverse} $\mu_h=\lambda_g^{-1}$.

        \item The map $\mu_h$ is bijective.

        \item \label{lemma:mu-inverses} $\mu_{hs}(s^{-1}) = \mu_h(s)^{-1}$ for all $s \in S(\tX)$.
    \end{enumerate}
\end{lemma}

\begin{proof}
    \begin{enumerate}[label=(\roman*)]
    \item Let $w \in X_S^*$ with $\overline{w}=g$ and $\varphi(g)=\overline{\lambda_1(w)}$ (by Proposition \ref{lemma:1-stays-1}\ref{lemma:map-phi-well-def}). The map $\lambda_1$ sends each letter $x_v^{\pm 1}$ of $w$ to $\tx_{v_0}$ or $\tx_{v_1}$ and all other letters to other letters. It is thus clear that the parity is preserved, hence $\teta(\varphi(g)) \equiv \eta_v(g) \mod 2$. Then, it follows using \eqref{eq:lambda-inverse} that
    \begin{align*}
        \mu_h = \mu_{\varphi(g)} = \lambda_g^{-1}.
    \end{align*}

    \item It follows from \ref{lemma:mu-is-lambda-inverse}.

    \item Let $z:=\mu_h(s) = \lambda_g^{-1}(s)$, where the second equality is by \ref{lemma:mu-is-lambda-inverse}. Then,
    \begin{align*}
        hs = \varphi(g)\lambda_g(z) = \varphi(gz),
    \end{align*}
    hence $\mu_{hs} = \lambda_{gz}^{-1}$. By Lemma \ref{lemma:lambda}\ref{lemma:lambda-2}, 
    \begin{align*}
        \mu_{hs}(s^{-1}) = \lambda_{gz}^{-1}(s^{-1}) = z^{-1} = \mu_h(s)^{-1}.
    \end{align*}
    


    \end{enumerate}
\end{proof}

Fix $h \in \td$. As for $\lambda$, the map $\mu_h$ extends to words in $S(\tX)^*$ by relabelling each letter from the vertex at which it is read. Let $y=(y_1,\ldots,y_n)\in S(\tX)^*$, then
\begin{align*}
    \mu_h(y):=(\mu_h(y_1),\mu_{hy_1}(y_2),\ldots,\mu_{hy_1\cdots y_{n-1}}(y_n)).
\end{align*}
By definition, $\mu_h(y_1y_2)=\mu_h(y_1)\,\mu_{h\overline{y_1}}(y_2)$ for all $y_1,y_2\in S(\tX)^*$. Moreover, by iteratively applying Lemma \ref{lemma:mu-properties}\ref{lemma:mu-inverses}, we get
\begin{align} \label{eq:mu-inverse-on-words}
    \mu_h(y^{-1})=(\mu_{h\overline{y}^{-1}}(y))^{-1}.
\end{align}

\begin{proposition} \label{lemma:1-stays-1-part-2}
\begin{enumerate}[label=(\roman*)]
    \item \label{item:closed-word-1-SX} Let $y \in S(\tX)^*$ such that $\overline{y} = 1$ in $\td$, then for all $h \in \td$
    \begin{align*} 
        \overline{\mu_{h}(y)} = 1 \text{ in } D.
    \end{align*}

    \item \label{lemma:map-mu-well-def}
    The map
   \begin{align*}
    \psi : \cy(\td, S(\tX)) & \to \cy(D,X_S) \\
    hs & \mapsto \psi(h)\mu_h(s),
\end{align*}
where $h \in \td, s \in S(\tX)$ with $\psi(1)=1$
is well defined.
    \end{enumerate}
\end{proposition}

\begin{proof}
\begin{enumerate}[label=(\roman*)]
\item The proof follows the same ideas as Proposition \ref{lemma:1-stays-1}\ref{item:closed-word-1-XS}, but for words representing $1$ which build the relations of $\td$.

\item Analogous to Proposition \ref{lemma:1-stays-1}\ref{lemma:map-phi-well-def}, using \ref{item:closed-word-1-SX} and \eqref{eq:mu-inverse-on-words}.
\end{enumerate}
\end{proof}

\begin{theorem}[Theorem \ref{thm:intro-iso}] \label{thm:cay-xs-sx}
    Let $D$ be a Dyer group and $\td$ its infinite-vertex double. Then, $\cy(D,X_S)$ and $\cy(\td, S(\tX))$ are isomorphic as graphs.
\end{theorem}

\begin{proof}
We show that the maps $\varphi$ and $\psi$ defined in Propositions \ref{lemma:1-stays-1} and \ref{lemma:1-stays-1-part-2} are mutually inverse.
Let $g \in D, g=z_1\ldots z_n$ for $z_i \in X_S$. Then,
\begin{align*}
    \varphi(g) & = \lambda_1(z_1) \lambda_{z_1}(z_2) \ldots \lambda_{z_1 \ldots z_{n-1}} (z_n) \\
    \implies \psi(\varphi(g)) & = \mu_{\varphi(1)}(\lambda_1(z_1)) \ldots \mu_{\varphi(z_1 \ldots z_{n-1})}(\lambda_{z_1 \ldots z_{n-1}}(z_n)) \\
                                & = z_1 \ldots z_n = g,
\end{align*}
by Lemma \ref{lemma:mu-properties}\ref{lemma:mu-is-lambda-inverse}. Hence, $\psi \circ \varphi = \id_D$ and since $\varphi$ is surjective by Proposition \ref{lemma:1-stays-1}\ref{lemma:map-phi-well-def}, it is bijective with inverse $\psi$.

Finally, $\varphi$ and $\psi$ are graph isomorphisms since $\varphi$ maps the edge $\{g,gz\}$ to the edge $\{\varphi(g),\varphi(g)\lambda_g(z)\}$ for $g \in D, z\in X_S$, and $\psi$ maps the edge $\{h,hs\}$ to the edge $\{\psi(h),\psi(h)\mu_h(s)\}$, for $h \in \td, s \in S(\tX)$. 
\end{proof}

\subsection{The FFTP} \label{subsec:xs-fftp}
\begin{theorem} \label{thm:FFTP}
    Let $D = D(\GG, m, f)$ be a Dyer group. Then in $\cy(D,S(X))$, every non-geodesic path is $(2Q+2)$-fellow-travelled by a shorter path with the same endpoints, where $Q$ is the constant from the Parallel Wall Theorem \ref{thm:pwt-dyer}.
\end{theorem}

The main idea of the proof of this theorem is to generalise Noskov's proof of the FFTP for Coxeter groups to the setting of Dyer groups; see \cite[Theorem 4]{noskov}. We first introduce a technical lemma and an additional definition.

\begin{lemma} \label{lemma: sectors-of-geodesic}
    Let $\gamma = \{x_0, \ldots, x_n\}$ be a geodesic path in $\caydsx$. Let $H \in \HH$ and $0 \leq i < j < k \leq n$. If $x_i$ and $x_k$ lie in the same sector of $H$, then $x_j$ also lies in this sector of H.
\end{lemma}

\begin{proof}
    Assume for contradiction that $x_j$ lies in a distinct sector of $H$ from $x_i$ and $x_k$. Then, the subpaths $\eval{\gamma}_{[x_i, x_j]}$ and $\eval{\gamma}_{[x_j, x_k]}$ of $\gamma$ both cross the hyperplane $H$. This implies $\gamma$ must cross the hyperplane $H$ twice, which contradicts $\gamma$ being a geodesic path \cite[Theorem 3.9(iii)]{genevois}.
\end{proof}

\begin{definition}
    A non-geodesic path $\gamma = \{x_0, \ldots, x_n\}$ is said to be \textit{minimal} if it crosses exactly one hyperplane twice, in its first and last edges, $\{x_0, x_1\}$ and $\{x_{n-1}, x_n\}$.
\end{definition}

\begin{proof}[Proof of Theorem \ref{thm:FFTP}]
    Let $\gamma = \{x_0, \ldots, x_n\}$ be a non-geodesic path in $\caydsx$ and
let $\gamma' := \eval{\gamma}_{[x_i, x_j]}$ be a non-geodesic subpath of $\gamma$ of
minimal length. By
\cite[Theorem 3.9(iii)]{genevois}, there exists a hyperplane $H\in\HH$ that $\gamma'$ crosses at
least twice. Since every proper subpath of $\gamma'$ is geodesic, they cross each hyperplane at most once, thus $H$ is crossed
exactly twice in the first and last edges of $\gamma'$ and $\gamma'$ is a minimal non-geodesic path.

We now show that $\gamma'$ lies in a $Q$-neighbourhood of $H$. Consider the subpath
$\gamma'':=\eval{\gamma}_{[x_{i+1},x_{j-1}]}$; it is geodesic. Assume for contradiction that there exists a vertex $x_k$ of $\gamma'$
such that $d(x_k, M(H)) > Q$. Since the edges
$\{x_i,x_{i+1}\}$ and $\{x_{j-1},x_j\}$ lie in $H$,
$x_i,x_{i+1},x_{j-1},x_j\in M(H)$, and we have
$i+1<k<j-1$. By the Parallel Wall Theorem~\ref{thm:pwt-dyer}, there exists a
hyperplane $J$, non-transverse to $H$, which separates $x_k$ from $H$. By
Lemma~\ref{lemma:transversality}(iii), $J\neq H$ and $x_k$ does not lie in the
sector of $J$ containing $N(H)$. 
Since
$x_{i+1}, x_{j-1} \in M(H) \subseteq N(H)$, this implies that $x_k$ lies in a different sector of $J$ from $x_{i+1}$ and $x_{j-1}$. This contradicts
Lemma~\ref{lemma: sectors-of-geodesic}, applied to the geodesic path $\gamma''$, the
hyperplane $J$ and the vertices $x_{i+1}, x_k, x_{j-1}$. Hence
\begin{equation}\label{eq:distance-bound}
    d(x_k, M(H)) \leq Q \qquad\text{for all } i+1 \leq k \leq j-1.
\end{equation}
        
Consider the edges $e_{i+1}=\{x_i, x_{i+1}\}$ and $e_{j}=\{x_{j-1}, x_{j}\}$ of $\gamma$ and let $C_{i+1},C_j$ be the cliques containing them, respectively. Since both edges lie in $H$, the two cliques also lie in the hyperplane $H$. Recall that by \cite[Theorem 3.9(i)]{genevois}, every sector of $H$ contains exactly one vertex of each clique of $H$, and that $\stbr(H) = \langle \rho_H \rangle $ permutes freely and transitively the sectors of $H$ by Lemma \ref{lemma:free-transitive-action}. On the one hand, there is a power $\alpha_i$ such that $\rho_H^{\alpha_i}x_{i+1}$ lies in $x_i$'s sector of $H$, and it follows from Definition \ref{def:rotstab} and Lemma \ref{lemma:cliques}\ref{lemma:cliques-same-stab} that $\rho_H^{\alpha_i}x_{i+1} = x_i$. On the other hand, $\rho_H^{\alpha_i}x_{j-1}$ belongs to the clique $C_j$ and so satisfies $d(\rho_H^{\alpha_i}x_{j-1}, x_j) \leq 1$. \\
    Finally, consider the path 
    \begin{align*}
\delta = \{x_0, \ldots, x_{i-1}, x_i=\rho_H^{\alpha_i}x_{i+1}, \rho_H^{\alpha_i}x_{i+2}, \ldots, \rho_H^{\alpha_i}x_{j-1}=x_j, x_{j+1}, \ldots, x_n\},   
    \end{align*}
if $d(\rho_H^{\alpha_i}x_{j-1}, x_j) = 0$, or 
\begin{align*}
    \delta = \{x_0, \ldots, x_{i-1}, x_i=\rho_H^{\alpha_i}x_{i+1}, \rho_H^{\alpha_i}x_{i+2}, \ldots, \rho_H^{\alpha_i}x_{j-1}, x_j, x_{j+1}, \ldots, x_n\},
\end{align*}
if $d(\rho_H^{\alpha_i}x_{j-1}, x_j) = 1$. Then, $\ell(\delta) \leq \ell(\gamma)-1$ and the distance between them is computed as follows. In the first case,
    \begin{align*}
        d(\gamma(t), \delta(t)) \leq \begin{cases}
            0, & \text{ if } 0 \leq t \leq i \\
            d(x_{t}, \rho_H^{\alpha_i}x_{t+1}), & \text{ if } i+1 \leq t \leq j-2 \\
            2, & \text{ if } j-1 \leq t \leq n-2 \\
            1, & \text{ if } t = n-1 \\
            0, & \text{ if } t = n 
        \end{cases}
    \end{align*}
    and in the second case,
    \begin{align*}
    d(\gamma(t), \delta(t)) \leq \begin{cases}
            0, & \text{ if } 0 \leq t \leq i \\
            d(x_{t}, \rho_H^{\alpha_i}x_{t+1}), & \text{ if } i+1 \leq t \leq j-2 \\
            1, & \text{ if } j-1 \leq t \leq n-1 \\
            0, & \text{ if } t = n
        \end{cases}.
    \end{align*}
    It remains to compute $d(x_{t}, \rho_H^{\alpha_i}x_{t+1})$ for $i+1 \leq t \leq j-2$. Let $y \in M(H)$ be such that $d(y, x_{t+1}) \leq Q$; this is possible by \eqref{eq:distance-bound}. Then:
    \begin{align*}
        d(x_{t}, \rho_H^{\alpha_i}x_{t+1}) & \leq d(x_{t}, y) + d(y, \rho_H^{\alpha_i}y) + d(\rho_H^{\alpha_i}y, \rho_H^{\alpha_i}x_{t+1}) \\
        & \leq d(x_{t}, x_{t+1}) + d(x_{t+1}, y) + d(y, \rho_H^{\alpha_i}y) + d(\rho_H^{\alpha_i}y, \rho_H^{\alpha_i}x_{t+1}) \\
        & \leq 1 + Q + 1 + Q = 2Q+2.
    \end{align*}
 The inequality $d(y, \rho_H^{\alpha_i}y) \leq 1$ holds because $y$ lies in a clique of $H$, which is stabilised by $\langle \rho_H \rangle$.
We conclude that the paths $\gamma$ and $\delta$ fellow travel with constant $2Q+2$, and satisfy $\ell(\delta) < \ell(\gamma)$.
    \end{proof}

Finally, putting together Theorems \ref{thm:FFTP} and \ref{thm:cay-xs-sx}, we get the following corollary.
\begin{corollary}[Theorem \ref{thm:intro-fftp}]
\label{cor:FFTP}
    Let $D=D(\GG,m,f)$ be a Dyer group. Then, $(D,X_S)$ has the FFTP with constant $2\Tilde{Q}+2$, where $\Tilde{Q}$ is the constant from the Parallel Wall Theorem for $(\td, S(\tX))$.
\end{corollary}

\begin{proof}
    Let $\td$ be the Dyer group obtained by performing the infinite-vertex doubling operation. By Theorem \ref{thm:FFTP} applied to $\td$, which has no vertex labelled by $\infty$, $\cy(\td, S(\tX))$ satisfies the FFTP with constant $2\Tilde{Q}+2$, where $\Tilde{Q}$ is the constant from the Parallel Wall Theorem for $(\td, S(\tX))$. By Theorem \ref{thm:cay-xs-sx}, $\cy(D,X_S)$ is isomorphic to $\cy(\td, S(\tX))$, so the same property holds for $\cy(D,X_S)$.
\end{proof}

We conclude with the finiteness result on cone types, which follows from Theorem \ref{thm:NS}.
\begin{corollary}
\label{cor:cannon-pair}
    Let $D=D(\GG,m,f)$ be a Dyer group. Then, $(D,X_S)$ 
    is a Cannon pair.
\end{corollary}

 As mentioned, the FFTP is sensitive to the generating set, and it is an open problem whether $(D,X^{\pm 1})$ has
 finitely many cone types in the case where some vertex is labelled $4 \leq f(v) < \infty$ and $\GG$ has an edge labelled by at least $3$. 
 We nevertheless add that by a result of Antolín and Ciobanu \cite[Proposition $3.2$]{ACrelhyp}, one can build on $X_S$ to exhibit infinitely many generating sets for which a Dyer group has the FFTP.

\section*{Acknowledgments}
I would like to thank my advisor, Tatiana Nagnibeda, for her encouragement in pursuing this topic and for her guidance and support.
I am also grateful to Luis Paris, Murray Elder and Laura Ciobanu for introducing me respectively and chronologically to Dyer groups, the falsification by fellow-traveller property and mediangle graphs, and for the interesting and enriching discussions around these topics. Finally, I thank Anthony Genevois for pointing out a gap in the proof in the previous version of this paper and for sharing his counterexample with me.
This work was supported by the Swiss NSF grant 200020-200400 and the NCCR SwissMAP.

\bibliographystyle{abbrv}
\bibliography{refs}

\end{document}

%% file: dyer-gp-example.tex

\begin{tikzpicture}

\filldraw[black] (-1.5,0) circle (1.5pt);
\filldraw[black] (-1.5,0) node[anchor=south]{\scalebox{1}{$\infty$}};
\filldraw[black] (0,0) circle (1.5pt);
\filldraw[black] (0,0) node[anchor=south]{\scalebox{1}{$2$}};
\filldraw[black] (1.5,0) circle (1.5pt);
\filldraw[black] (1.5,0) node[anchor=south]{\scalebox{1}{$2$}};

\filldraw[black] (-0.75,0) node[anchor=north]{\scalebox{1}{$2$}};
\filldraw[black] (0.75,0) node[anchor=north]{\scalebox{1}{$m$}};

\draw[black] (-1.5,0) -- (0,0) -- (1.5,0);

\end{tikzpicture}

%% file: fftp.tex

\begin{tikzpicture}

\draw [black] plot [smooth, tension=1] coordinates { (-1,0.5) (0,1.2) (1,1) (2,1.2) (2.5,0.3)};
\filldraw[black] (1,1) node [anchor=south]{\scalebox{0.8}{$ \gamma$}};

\draw [black] plot [smooth, tension=1] coordinates { (-1,0.5) (0,0.2) (1,0.3) (2,0.5) (2.5,0.3)};
\filldraw[black] (1,0.3) node [anchor=north]{\scalebox{0.8}{$ \delta$}};

\filldraw[black] (-1,0.5) circle (0.5pt);
\filldraw[black] (2.5,0.3) circle (0.5pt);

\draw [black, dashed, line width=0.1mm] plot [smooth, tension=1] coordinates { (0,1.15) (0,0.25) };
\filldraw[black] (0,0.45) node [anchor=north]{\scalebox{.25}{$\bigvee$}};
\filldraw[black] (0,0.95) node [anchor=south]{\scalebox{.25}{$\bigwedge$}};
\filldraw[black] (0,0.70) node [anchor=west]{\scalebox{.5}{$ \leq k$}};

\draw [black, dashed, line width=0.1mm] plot [smooth, tension=1] coordinates { (2,1.1) (2,0.5) };
\filldraw[black] (2,0.77) node [anchor=north]{\scalebox{.25}{$\bigvee$}};
\filldraw[black] (2,0.9) node [anchor=south]{\scalebox{.25}{$\bigwedge$}};
\filldraw[black] (1.9,0.8) node [anchor=west]{\scalebox{.5}{$ \leq k$}};

\end{tikzpicture}